\documentclass[10pt]{amsart}
\usepackage{amssymb,graphics, mathtools, old-arrows, lmodern}
\newcommand{\Int}{\mathrm{int}}

\newtheorem{thm}{Theorem}[section]
\newtheorem{theorem}[thm]{Theorem} 
\newtheorem{notation}[thm]{Notation}

\newtheorem{cor}[thm]{Corollary}
\newtheorem{prop}[thm]{Proposition}
\newtheorem{con}[thm]{Conjecture}
\newtheorem{remark}[thm]{Remark}

\theoremstyle{definition}
\newtheorem{definition}[thm]{Definition}

\usepackage[normalem]{ulem}
\usepackage[usenames,dvipsnames]{color}

\begin{document}

\title{Taut foliations from double-diamond replacements}
\date{\today}

\author{Charles Delman}
\address{Department of Mathematics\\
         Eastern Illinois University\\
         Charleston, IL 61920}
\email{cidelman@eiu.edu}
\author{Rachel Roberts}
\address{Department of Mathematics\\
         Washington  University\\
         St Louis, MO 63105}
\email{roberts@wustl.edu}


\thanks{This work was partially supported by National Science Foundation Grant DMS-1612475.}


\date{\today}

\begin{abstract} 
Suppose $M$ is an oriented 3-manifold with connected boundary a torus, and suppose $M$ contains a properly embedded, compact, oriented, surface $R$ with a single boundary component that is Thurston norm minimizing in $H_2(M, \partial M)$. We define a readily recognizable type of sutured manifold decomposition, which for notational reasons we call \emph{double-diamond taut}, and show that if $R$ admits a \emph{double-diamond taut} sutured manifold decomposition, then for every boundary slope except one, there is a co-oriented taut foliation of $M$ that intersects $\partial M$ transversely in a foliation by curves of that slope.  In the case that $M$ is the complement of a knot $\kappa$ in $S^3$, the exceptional filling is the meridional one; in particular, restricting attention to rational slopes, it follows that every manifold obtained by non-trivial Dehn surgery along $\kappa$ admits a co-oriented taut foliation.  As an application, we show that if $R$ is a Murasugi sum of surfaces $R_1$ and $R_2$, where $R_2$ is an unknotted band  with an even number $2m\ge 4$ of half-twists, then every manifold obtained by non-trivial surgery on $\kappa= \partial R$ admits a co-oriented taut foliation.

\end{abstract}

\maketitle

\section{Introduction}

Taut foliations have long played an important role in the study of 3-manifolds, informed by the relationship among various geometric and algebraic properties. We present here the relevant  definitions, known results, and conjectures. For simplicity of exposition, we assume all 3-manifolds are oriented and are not homeomorphic to $S^1\times S^2$.

\begin{definition}
Call a 3-manifold \emph{foliar} if it supports a co-oriented taut foliation.
\end{definition}

The presence of a taut foliation (co-oriented or not) in a closed 3-manifold imposes restrictions on its fundamental group:  it is not a non-trivial free product (indeed, $M$ is irreducible) \cite{Novikov,Reeb,Rosenberg}, and it contains elements of infinite order \cite{Haefliger, Novikov,GO}-- indeed, the manifold is covered by $\mathbb{R}^3$\cite{Palmeira}.

  Another well-studied property of groups is \emph{orderability}:  that is, an ordering that respects the group operation.  Focusing on transformations by group elements acting on the left:

\begin{definition}
A   nontrivial group $G$ is \emph{left-orderable} if its elements can be given a strict total ordering $< $ which is left invariant, meaning that $g < h$ implies $fg < fh$ for all $f,g,h \in G$.
\end{definition}

Another invariant of 3-manifolds is  Heegaard Floer   Homology, introduced by Ozsv\'{a}th and Szab\'{o} \cite{OzSz2, OzSz3}.  In particular, the Heegaard Floer homology of a rational homology sphere $Y$ satisfies $$\text{Rank}(\widehat{HF}(Y; \mathbb{Z}/2)) \geq |H_1(M,\mathbb{Z})|.$$  If equality holds, $Y$ is called an  \emph{L-space};  elliptic manifolds, such as lens spaces, are examples \cite{OzSz4}.  In contrast, Ozsv\'{a}th and Szab\'{o} proved foliar manifolds cannot be L-spaces \cite{OzSz,Bowden,KR1,KR2}.

Boyer, Gordon, and Watson have conjectured that an irreducible rational homology 3-sphere is an L-space if and only if its fundamental group is not left-orderable and also asked if the presence of a \emph{co-oriented} taut foliation implies that a 3-manifold's fundamental group is left-orderable \cite{boyergordonwatson}, leading to the   \emph{L-space Conjecture}:

\begin{con}[L-space Conjecture \cite{OzSz3,boyergordonwatson,juhasz3}]
Let $Y$ be an irreducible rational homology 3-sphere.  Then the following three statements are equivalent:

\begin{enumerate}
\item $Y$ is not an L-space,
\item $\pi_1(Y)$ is  left-orderable, and
\item $Y$ is foliar.
\end{enumerate}
\end{con}

Namely, in the spirit of Haken, Waldhausen, and Thurston \cite{Haken, Waldhausen, thurstonsurvey}, these are three measures of what it means for a 3-manifold to be ``large,"  and the conjecture is that they are all equivalent.  (We note that Gabai includes co-orientability in his original definition of \emph{taut}.)  It is known only that (3) implies (1).  This paper addresses the conjecture that (1) implies (3).  

\begin{definition}
Let $M$ be a 3-manifold with boundary a torus.  A \emph{boundary slope} is an isotopy class of curves in $\partial M$.  
\end{definition}

Following Boyer and Clay we introduce:

\begin{definition}[\cite{boyerclay}]
A foliation $\mathcal{F}$ \emph{strongly realizes} a slope if $\mathcal{F}$ intersects   $\partial M$ transversely in a foliation by curves of that slope.
\end{definition}

 Note that if a foliation strongly realizes a compact slope, it may be capped off by disks to obtain a co-oriented taut foliation in the manifold obtained by Dehn filling along that slope. 

Suppose $\partial M$ is a torus and $M$ contains a properly embedded, compact, oriented, Thurston norm minimizing \cite{thurston} surface $R$ with a single boundary component.  Call the isotopy class of $\partial R$ the \emph{longitude} and any slope having geometric intersection number one with this longitude a \emph{meridian}.  In this paper, we define an easily-recognized criterion that implies such a 3-manifold supports co-oriented taut foliations strongly realizing all boundary slopes except a single  \emph{distinguished meridian} $\nu$.    Thus, the manifolds obtained by Dehn surgery along all slopes other than $\nu$ are foliar.  

Focusing attention on manifolds obtained by Dehn surgery on knots, we introduce the following notation, definition, and conjecture:

\begin{notation} Recall that a knot in a 3-manifold  $P$   is a submanifold of the interior of    $P$  homeomorphic to $S^1$. Given a knot $\kappa$ in   a closed manifold $P$, let $N(\kappa)$ denote an open regular neighborhood of $\kappa$ with closure contained in the interior of   $P$, and let  $X_{\kappa}$ denote the knot complement   $P\setminus N(\kappa)$.    In   the case $M = X_\kappa$, we show in the proof of Corollary \ref{main theorem} that the distinguished  meridian  indicated above  is the meridian of $\kappa$, namely the unique slope $\mu$ whose elements bound disks in the closure of $N(\kappa)$.
\end{notation}

\begin{definition}[\cite{DR1}]
A knot $\kappa$ in a 3-manifold   is \emph{persistently foliar} if, for each non-meridional boundary slope of $X_{\kappa}$, there is a  co-oriented taut foliation strongly realizing that slope.  
\end{definition}

Many knots in $S^3$ are known to be persistently foliar, including all alternating and Montesinos knots with no L-space surgeries (\cite{DR2} and \cite{DR3, DR4}, respectively),  all fibered knots with  fractional Dehn twist coefficient zero \cite{Rfib2,DR1}, and all composite knots in which each of two summands is alternating, Montesinos, or fibered \cite{DR1}.  (Among the alternating and Montesinos knots, only the $(2,n)$-torus knots  \cite{OzSz4} and $(-2, 3, n)$-pretzel knots, where $n$ is (odd and) positive   \cite{LM16, BM18}, have L-space surgeries.)   As a direct derivative of the L-space conjecture, we have the following pair of conjectures about knots:

\begin{con}[Classical L-space knot conjecture]
A knot in $S^3$ is persistently foliar if and only if it has no non-trivial L-space or reducible surgeries. 
\end{con}

\begin{con}[L-space knot conjecture]
A knot in an L-space is persistently foliar if and only if it has no non-trivial L-space or reducible surgeries. 
\end{con}

In Section \ref{spines and branched surfaces} we provide all essential background on spines and branched surfaces.  In Section \ref{sutured manifolds} we provide the necessary definitions and results from sutured manifold theory.  Some familiarity with laminations and foliations is assumed, or would at least be helpful to the reader.   In Section \ref{modifications}, we prove a technical theorem (Theorem \ref{thm:Btodoublediamond}) about a particular type of sutured manifold decomposition with associated distinguished meridian, which we call \emph{double-diamond taut}.  We defer the statement of this theorem and note the following  direct consequence:

\medskip

\noindent \textbf{Corollary \ref{main theorem}.} \emph{Suppose the sutured manifold decomposition $\left(M|_R, \partial M|_{\partial R}\right)\overset{S}{\rightsquigarrow} (M',\gamma')$ is double-diamond taut, with associated distinguished meridian $\nu$.  If  the branched surface associated to this decomposition has no sink disk disjoint from $\partial M$, then there are co-oriented taut foliations  that strongly realize all boundary slopes except $\nu$.  Moreover, if $M=X_{\kappa}$ for some knot $\kappa\subset S^3$, $\kappa$ is persistently foliar.}

\medskip

  We end with an application to knots in $S^3$.  By  Ghiggini  \cite{Ghiggini}, Ni \cite{Ni1,Ni2} and Juhasz \cite{juhasz1,juhasz2}, non-fibered knots do not admit L-space surgeries;  hence, if the L-space  conjecture  holds, all non-fibered knots that do not admit a reducible surgery are persistently foliar.  If in addition the cabling conjecture holds, then by a result of Scharlemann \cite{scharl}  all non-fibered hyperbolic knots are persistently foliar. Although knots with Seifert surface obtained by plumbing Hopf bands have long been known to be fibered \cite{Stallings}, knots obtained by plumbing a Seifert surface with a $(2,2n)$ torus link with $|n| \geq 2$ are never fibered \cite{Gabfib}.  Thus we expect such knots to be persistently foliar, and indeed this is the case:

\medskip

\noindent \textbf{Corollary \ref{plumbing corollary}.} \emph{Suppose $\kappa$ is a knot in $S^3$ with minimal genus Seifert surface $R$. If $R$ is a plumbing  of surfaces $R_1$ and $R_2$, where $R_2$ is a unknotted band with an even number $2m\ge 4$ of half-twists, then the decomposing disk for $R_2$ can be chosen to be  double-diamond taut. Thus, $\kappa$ is persistently foliar.}

\medskip

  \section{Acknowledgement}

   It is a pleasure and a privilege to honor Steven Boyer for his many contributions over the years. We thank the referee for providing a thorough report that helped us improve the exposition significantly. 

\section{Spines and branched surfaces}\label{spines and branched surfaces}

For completeness we include the relevant standard definitions and results on spines and branched surfaces, as also found in \cite{DR1}.

 \begin{notation}
Given any metric space $X$ and any subset $A \subset X$, the \emph{closed complement} of $A$ in $X$,  denoted $X|_A$,  is the metric completion of $X\setminus A$. 
\end{notation}
 
  \begin{remark}
The space $X|_A$ is obtained by ``cutting" $X$ along $A$.  Although other authors have used the notation $X_A$,  we feel the inclusion of a vertical slash evokes the notion of ``cutting."
\end{remark}

 \begin{remark}We introduce this   notation to unify several related concepts and simplify the notation that arises in examples;  however, the reader may prefer a different perspective.  In the context of this paper, $X$ will be a $3$-manifold and $A$ an embedded surface, or $X$ will be a surface and $A$ an embedded simple closed curve or arc.   For these examples, $X$ is retrieved as a quotient space of $X|_A$ by identifying corresponding points of the two copies of $A$.  Also, in these examples, $X|_A$ is homeomorphic to $X \setminus \Int N(A)$, where $N(A)$ is a regular neighborhood;    furthermore, if $N(A)$ is $I$-fibered (see Definition \ref{I-fibered}, which extends in the obvious way to simple curves in surfaces) collapsing fibers gives a quotient map from $X \setminus \Int N(A)$ to a manifold homeomorphic to $X$.  However, the definition above extends to laminations, for which regular neighborhoods do not apply, and our notation avoids the cumbersome notation of the regular neighborhood description.  \end{remark}

\begin{definition}
A {\it standard spine} \cite{C} is a space $\Sigma$ locally modeled on one of the spaces of Figure~\ref{spine}.   A standard spine with boundary has the additional local models shown in Figure~\ref{spineboundary}. The {\it critical locus}  $\Gamma$  of $\Sigma$ is the 1-complex of points of $\Sigma$ where the spine is not locally a manifold.  The critical locus is a stratified space (graph) consisting of triple points $\Gamma^0$ and arcs of double points $\Gamma^1 = \Gamma \setminus \Gamma^0$. The components of $\Sigma|_{\Gamma}$ are called the {\it sectors} of $\Sigma$. 
\end{definition}

\begin{figure}[ht]
\scalebox{.3}{\includegraphics{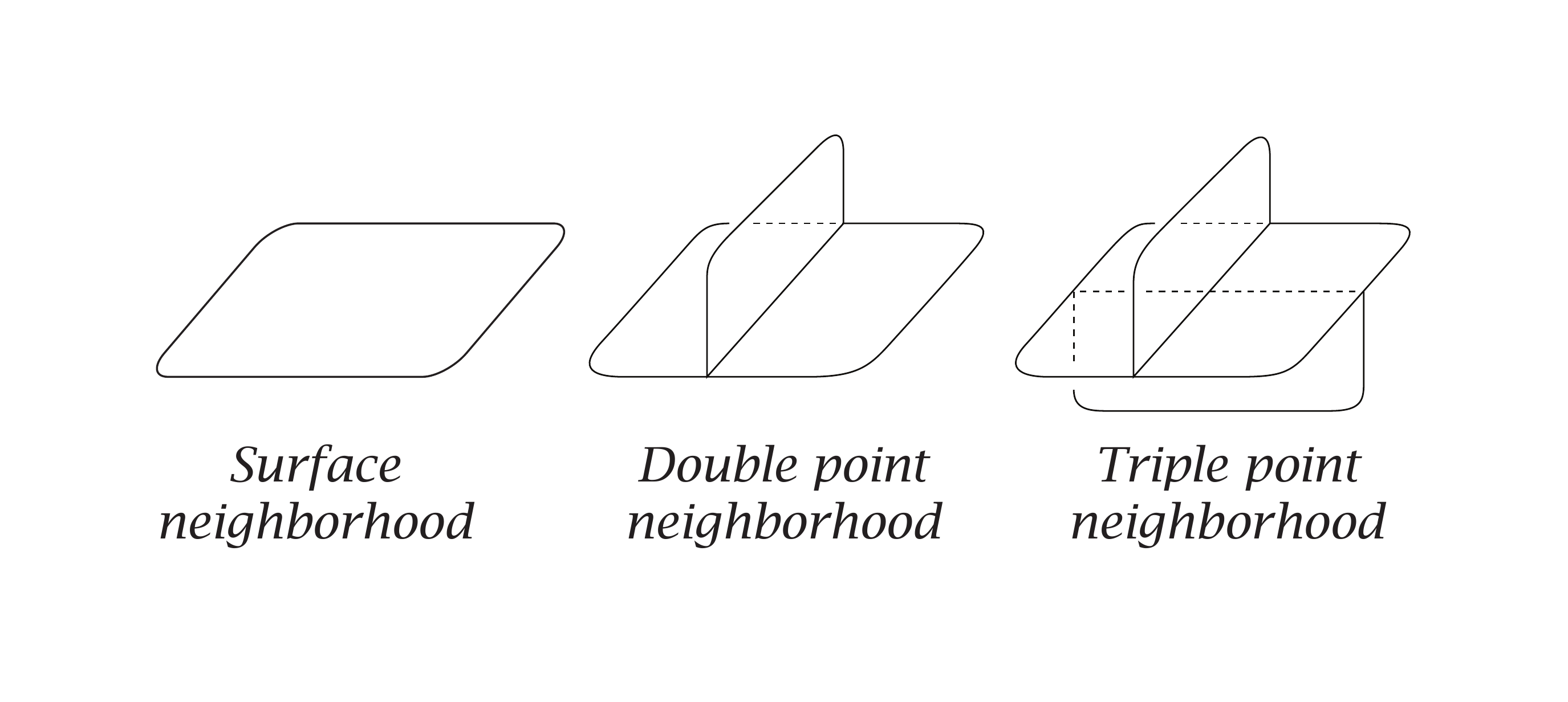}}
\caption{Local models of a standard spine at interior points.}\label{spine}
\end{figure} 

\begin{figure}[ht]
\scalebox{.3}{\includegraphics{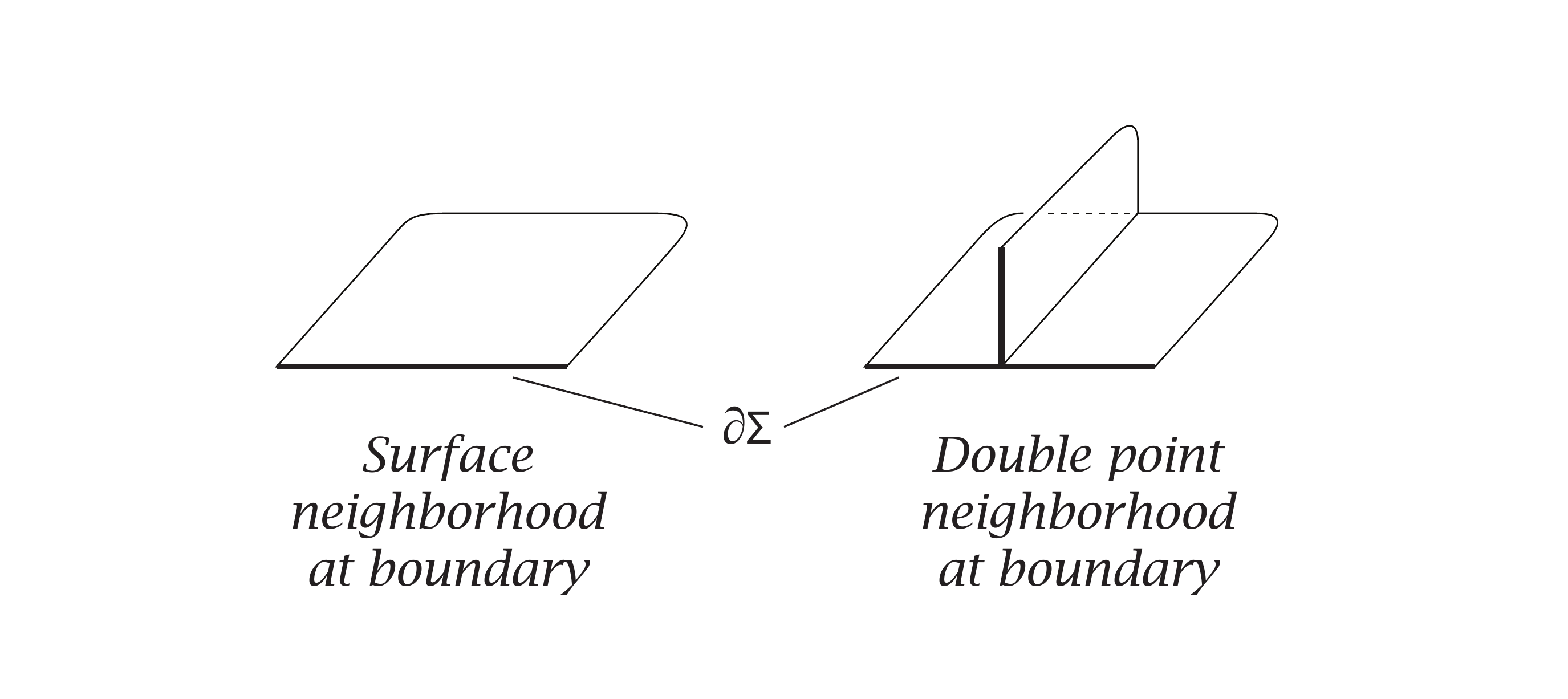}}
\caption{Local models of a standard spine at boundary points.}\label{spineboundary}
\end{figure}

\begin{definition} 
A {\it branched surface (with boundary)} (\cite{W}; see also \cite{Floyd-Oertel, O1,O2}) is a space $B$ locally modeled on the spaces of Figure~\ref{brsurf} (along with those in Figure~\ref{brsurfboundary}); that is, $B$ is homeomorphic to a spine, with the additional structure of a well-defined tangent plane at each point. The {\it branching locus} $\Gamma$ of $B$ is the 1-complex of points of $B$ where $B$ is not locally a manifold; such points are called \emph{branching points}.  The branching locus is a stratified space (graph) consisting of triple points $\Gamma^0$ and arcs of double points $\Gamma^1 = \Gamma \setminus \Gamma^0$. The components of $B|_\Gamma$ are called the {\it sectors} of $B$.  
\end{definition}

\begin{figure}[ht]
\scalebox{.3}{\includegraphics{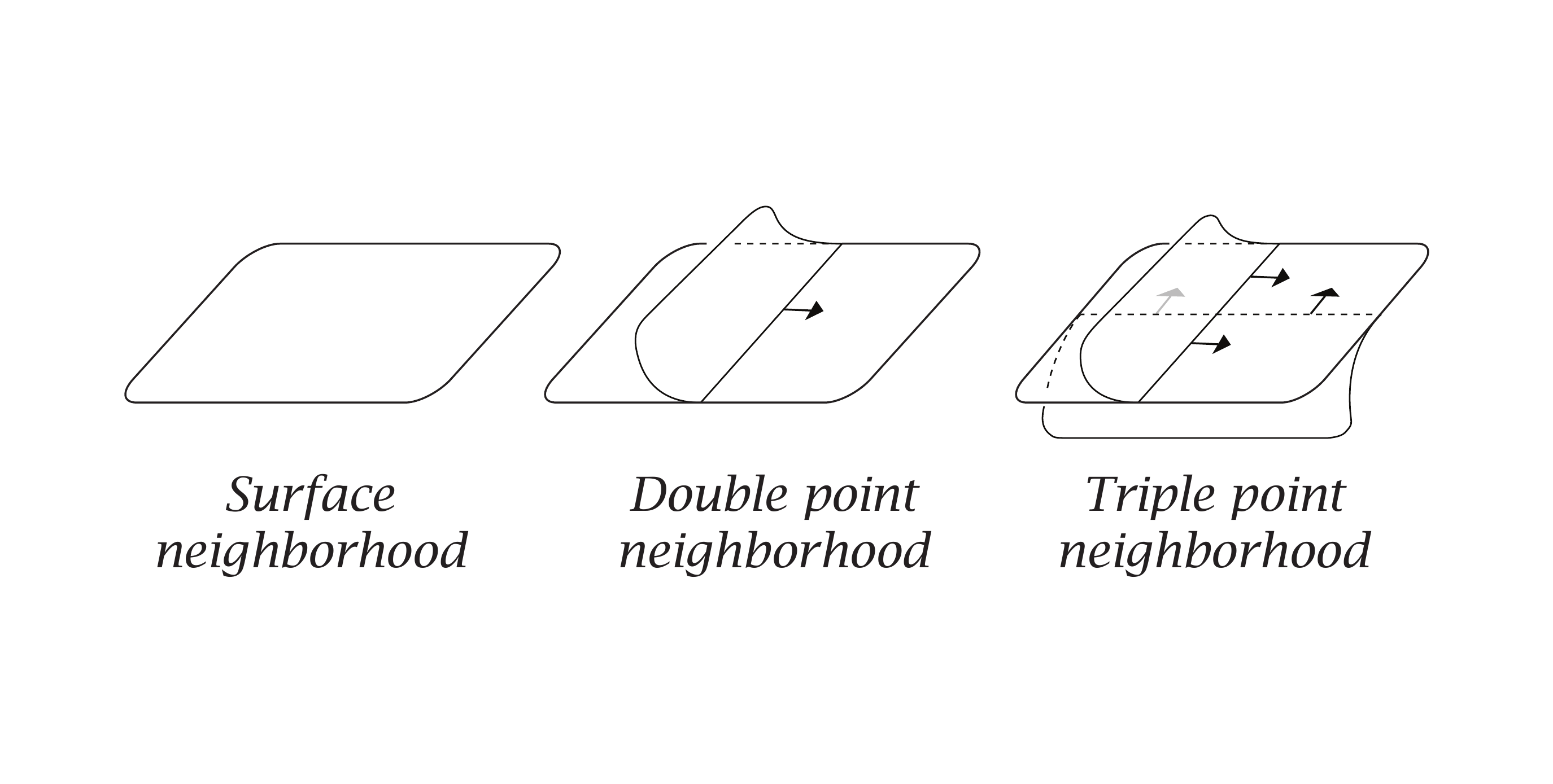}}
\caption{Local model of a branched surface at interior points.}\label{brsurf}
\end{figure}

\begin{figure}[ht]
\scalebox{.3}{\includegraphics{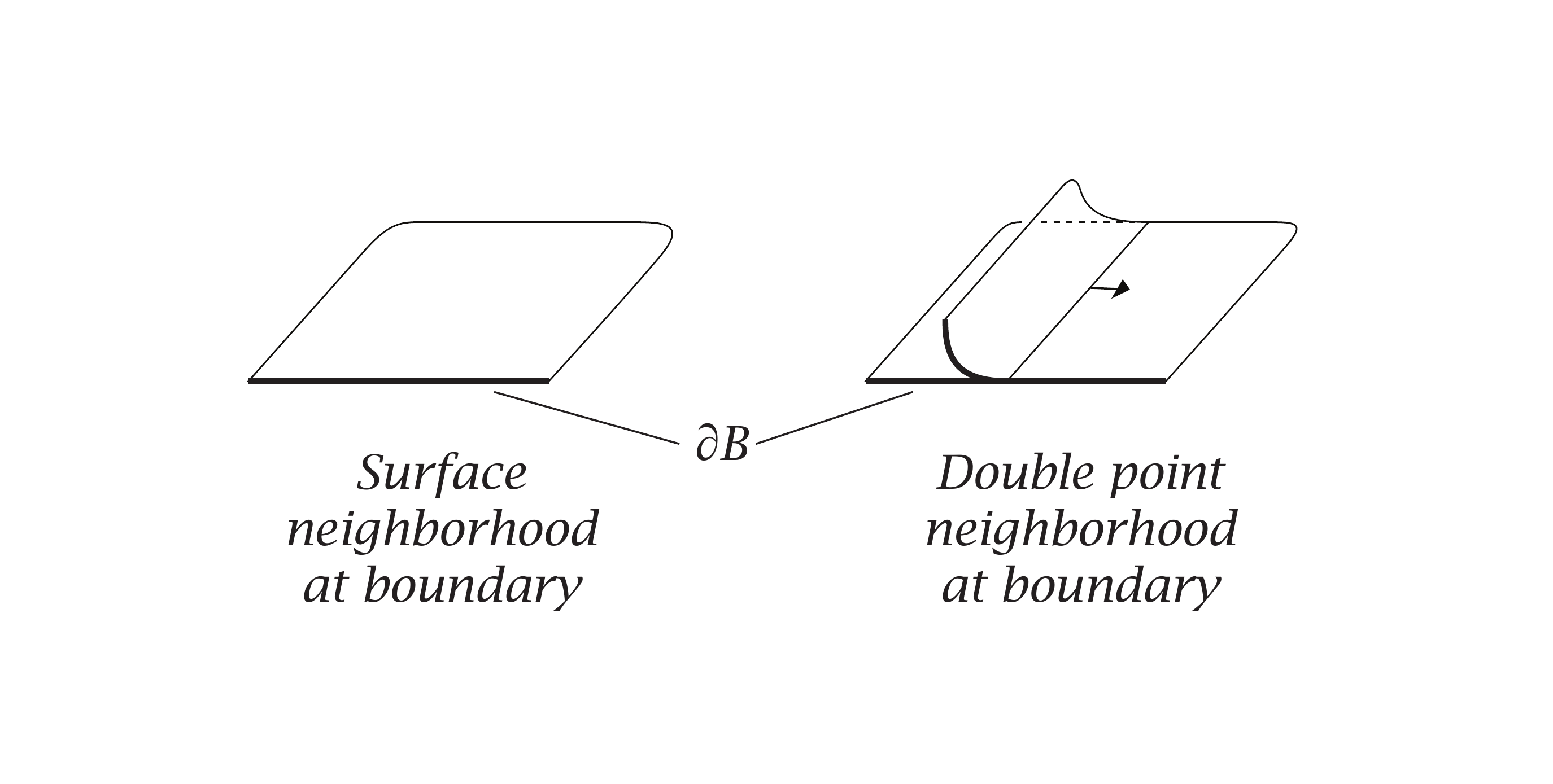}}
\caption{Local model of a branched surface at boundary points.}\label{brsurfboundary}
\end{figure}

\begin{definition}[\cite{Floyd-Oertel}] \label{I-fibered}
An \emph{$I$-fibered neighborhood} of a branched surface $B$  in a $3$-manifold $M$ is a regular neighborhood $N(B)$ foliated by interval fibers that intersect $B$ transversely, as locally modeled by the spaces in Figure \ref{figure: I-bundle neighborhood} at interior points;  if the ambient manifold $M$ has non-empty boundary, all spines and branched surfaces are assumed to be properly embedded, with $N(B)\cap\partial M$ a union (possibly empty) of I-fibers. The surface $\partial N(B) \setminus \partial M$ is a union of two subsurfaces,  $\partial_v N(B)$ and  $\partial_h N(B)$, where $\partial_v N(B)$, the \emph{vertical boundary},  is a union of sub-arcs of I-fibers, and $\partial_h N(B)$,  the \emph{horizontal boundary}, is everywhere transverse to the I-fibers.  
\end{definition}

\begin{definition}
The \emph{complementary regions} of $B$ are the components of $M \setminus \text{int}N(B)$;  the \emph{complement} of $B$ is the union of the complementary regions.
\end{definition}

\begin{figure}[ht]
\scalebox{.3}{\includegraphics{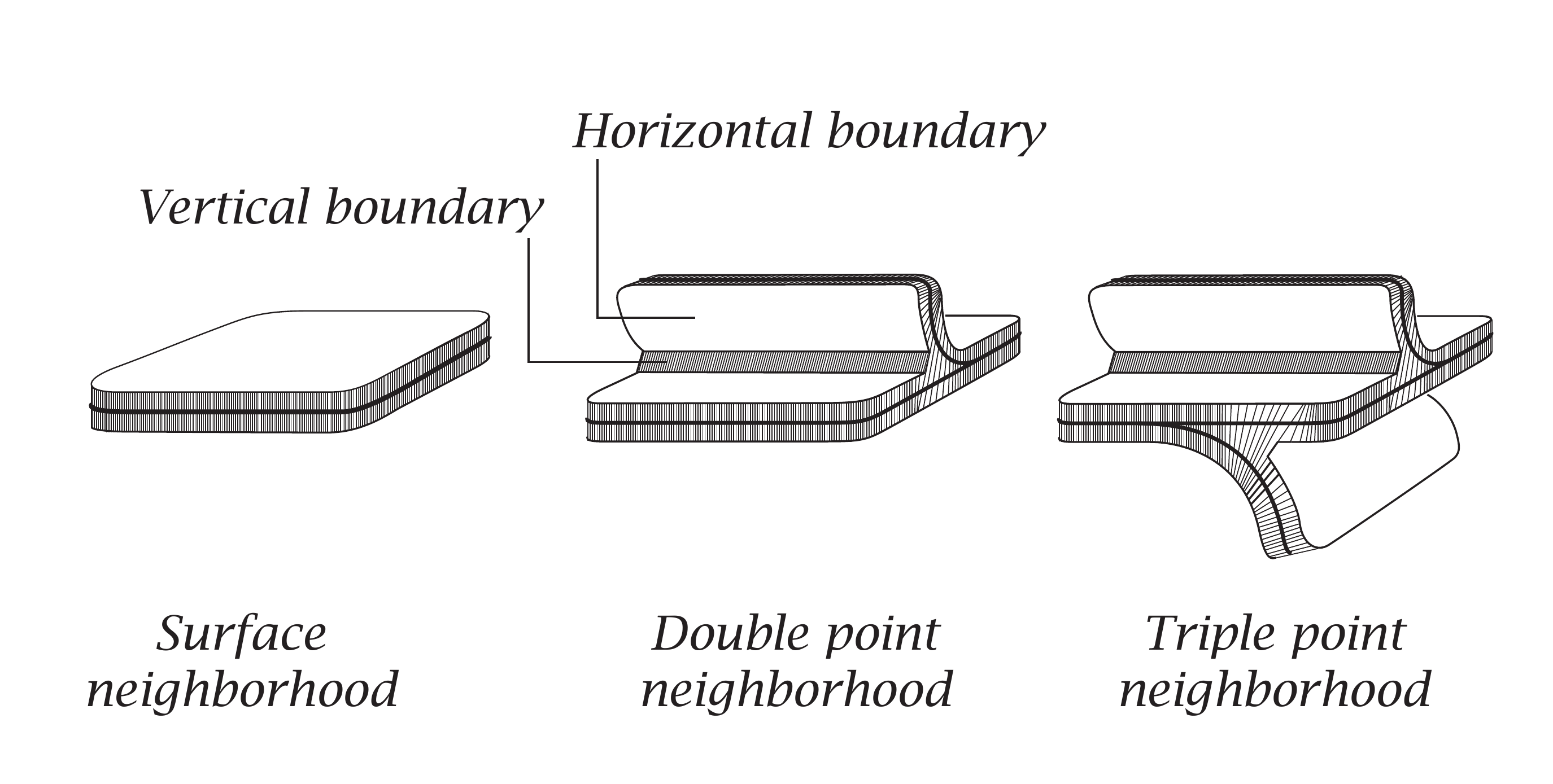}}
\caption{Local models for $N(B)$ (at interior points).}  \label{figure: I-bundle neighborhood}
\end{figure}

 \begin{definition}
A surface is  \emph{carried by} $B$ if it is contained in $N(B)$ and is everywhere transverse to the one-dimensional foliation of $N(B)$. A surface is \emph{fully carried by} $B$ if it carried by $B$ and has nonempty intersection with every I-fiber of $N(B)$. A lamination $\mathcal L$ is \emph{carried} by $B$ if each leaf of $\mathcal L$ is carried by $B$, and \emph{fully carried} if, in addition, each I-fiber of $N(B)$ has nonempty intersection with some leaf of $\mathcal L$.  A foliation is \emph{fully carried} by $B$ if there is a Denjoy splitting of $\mathcal F$ that is fully carried by $B$.
\end{definition}

Let $\pi$ be the retraction of $N(B)$ onto the quotient space obtained by collapsing each fiber to a point.  The branched surface $B$ is obtained, modulo a small isotopy, as the image of $N(B)$ under this retraction.  We will freely identify $B$ with this image \cite{O1} and the core of each component of vertical boundary with its image in $\Gamma$.   Double points of the branching locus are \emph{cusps} with \emph{cusp direction} pointing inward from the vertical boundary if $B$ is viewed as the quotient of $N(B)$ obtained by collapsing the vertical fibers to points.  

\begin{notation}
Cusp directions will be indicated by arrows, as in Figures \ref{brsurf}  and \ref{brsurfboundary}.  Call a sector     $S$  of $B$ a \emph{source} (\emph{sink}) if all cusp directions along $\partial S$  point out of (into)   $S$.
\end{notation}

\begin{definition}[\cite{GO}]\label{defn:GO} A branched surface $B$ in a  closed 3-manifold $M$ is called an {\it essential} branched surface if it satisfies the following conditions:
\begin{enumerate}
\item  $\partial_h N(B)$ is incompressible in $M\setminus \Int(N(B))$, no component of  $\partial_h N(B)$ is a sphere, and $M\setminus \Int(N(B))$ is irreducible.

\item There is no monogon in $M\setminus \Int(N(B))$; i.e., no disk $D\subset M\setminus \Int(N(B))$
with $\partial D = D\cap N(B) = \alpha\cup\beta$, where $\alpha\subset\partial_v N(B)$ is in an interval fiber of $\partial_v N(B)$ and $\beta\subset \partial_h N(B)$.

\item There is no Reeb component; i.e., $B$ does not carry a torus that bounds a solid torus in $M$.
\end{enumerate}
\end{definition}

In the spirit of earlier definitions of Gabai, Oertel, Sullivan and others, we introduce:

\begin{definition}
A branched surface is \emph{taut} if it is co-oriented, has taut sutured manifold complement (see Definition~\ref{defn:tautsm}), and through every sector there is a closed  oriented curve that is positively transverse to $B$.
\end{definition}

Observe that a taut branched surface  is, in particular, essential;  furthermore, if a taut branched surface fully carries a lamination, this lamination is a sub-lamination of a taut foliation (see Corollary \ref{cor:canextend} below).  However, in practice, it can be difficult to determine whether an essential branched surface fully carries a lamination.  In  \cite{Li0,Li}, Li defines the notion of \emph{laminar}, a very useful criterion that is sufficient (although not necessary) to guarantee that an essential branched surface fully carries a lamination. We recall the necessary definitions here.

\begin{definition}[\cite{Li0,Li}] Let $B$ be a branched surface in a 3-manifold $M$.   A  {\it sink disk}  is a  disk sector of $B$ that is a sink.  A {\it half sink disk} is a sink disk which has nonempty intersection with $\partial M$. 
\end{definition}

\begin{figure}[ht]
\scalebox{.2}{\includegraphics{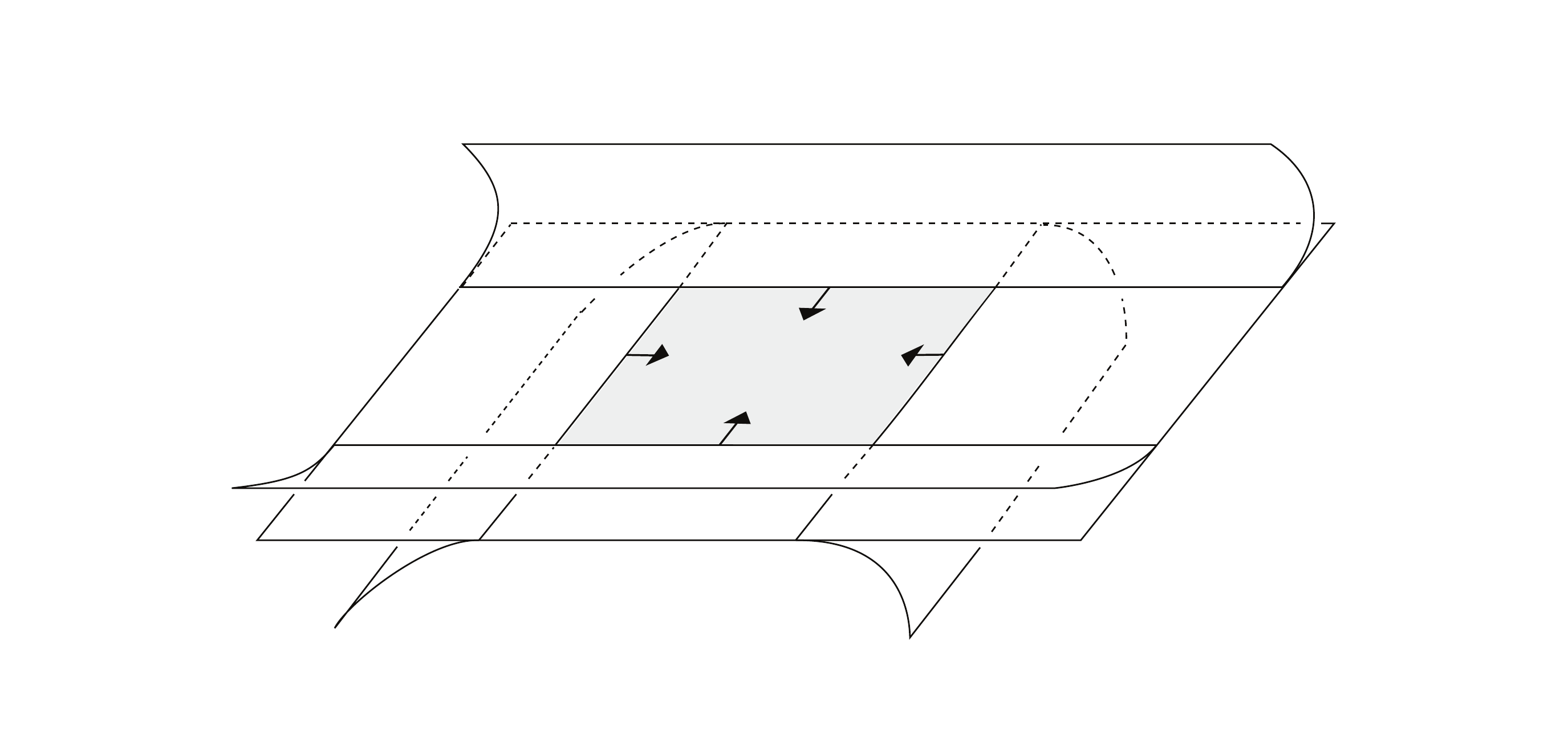}}
\caption{A sink disk.}\label{sinkdisk}
\end{figure}

Sink disks  and half sink disks play a key role in Li's notion  of  laminar branched surface. A   sink disk  or half sink disk  $D$ can be  eliminated by splitting $D$ open along a disk in its interior; these trivial splittings must be ruled out:

\begin{definition}[\cite{Li0}, \cite{Li}]
Let $D_1$ and $D_2$ be the two disk components of the horizontal boundary of a $D^2 \times I$ region in $M \setminus \Int(N(B))$. If the projection $\pi: N(B) \to B$ restricted to the interior of $D_1 \cup D_2$ is injective, i.e, the intersection of any $I$-fiber of $N(B)$ with $\Int(D_1) \cup \Int(D_2)$ is either empty or a single point, then we say that $\pi(D_1 \cup D_2)$ forms a \emph{trivial bubble} in $B$.
\end{definition}

\begin{definition}[\cite{Li0}, \cite{Li}]\label{defn:Li}
An essential branched surface $B$ in a compact 3-manifold $M$ is called  {\it laminar}   if it satisfies the following conditions:
\begin{enumerate}
\item $B$ has no trivial bubbles.
\item $B$ has no sink disk or half sink disk.
\end{enumerate}
\end{definition}

\begin{theorem} [\cite{Li0}, \cite{Li}]
Suppose $M$ is a compact and orientable 3-manifold. 
\begin{itemize}
\item[(a)] Every laminar branched surface in $M$ fully carries an essential lamination.
\item[(b)]Any essential lamination in $M$ that is not a lamination by planes is fully carried by a laminar branched surface.
\end{itemize}
\end{theorem}

\section{Sutured manifold decompositions} \label{sutured manifolds}

We assume the reader is familiar with the basics of co-oriented taut foliations. Precise definitions and terminology as used here can be found in   \cite{DR1}. The basics of Gabai's theory of sutured manifolds \cite{Gab1, Gab2, Gab3} play a key role in this paper, and in this section we give some necessary sutured manifold background. Where noted, we use the definitions of \cite{Gab1} as stated in \cite{CCGabai}.

\begin{definition}[Definition 2.6 of \cite{Gab1},\cite{CCGabai}]
 \label{sutman}
A  pair $(M,\gamma)$ is a \emph{sutured manifold} if $M$ is a compact, oriented $3$-manifold  and $\gamma\subset\partial M$ is a compact subsurface such that
$$
\gamma=A(\gamma)\cup T(\gamma),
$$
where $A(\gamma)\cap T(\gamma)=\emptyset$, $A(\gamma)$ is a disjoint  union of annuli and $T(\gamma)$ a disjoint union of tori.  The   components of $A(\gamma)$ are called   \emph{annular sutures}, the  components of  $T(\gamma)$ are called \emph{toral sutures}, and   the closure of $\partial M\setminus \gamma$ is denoted by $R(\gamma)$.    The surface $R(\gamma)$ comes with a transverse orientation;  let  $R_{+}(\gamma)$ be the subsurface oriented outward from $M$  and $R_{-}(\gamma)$ be  the inwardly oriented subsurface.  For each annular suture $A$, it is assumed that one component of $\partial A$ is also a component of $\partial R_{+}(\gamma)$, the other a component of $\partial R_{-}(\gamma)$. 
\end{definition}

As noted in \cite{Gab3}, a branched surface gives rise to a sutured manifold. If $B$ is a co-oriented branched surface in $M$, and $\partial M$ is a union of tori, then the regions $(\partial M\setminus\Int N(B),\partial_v N(B)\cap \partial M)$ are products.  Setting $\gamma = \partial_vN(B) \cup (\partial M\setminus \text{int}N(B))$, the pair $(M \setminus \Int N(B), \gamma)$ is a sutured manifold, with $R_+(\gamma)$ (respectively, $R_-(\gamma)$) consisting of the components of $\partial_hN(B)$ with co-orientation pointing into (respectively, out of) N(B).    

\begin{definition}[\cite{Gab1}]  A sutured manifold $(M,\gamma)$ is a \emph{product sutured manifold} if $(M,\gamma)$ is homeomorphic to $(F\times [0,1],\partial F\times [0,1])$, where $F$ is a compact, not necessarily connected, surface with nonempty boundary.
\end{definition}

\begin{definition}[\cite{Gab1, CCGabai}] \label{defn:tautsm}
A sutured manifold $(M,\gamma)$ is \emph{taut}  if $M$ is irreducible and $R(\gamma)$ is both  incompressible and Thurston norm minimizing \cite{thurston} in $H_{2}(M,\gamma)$. 
\end{definition}

In particular, product sutured manifolds are taut, as are sutured manifolds of the form $(M,\gamma)=\left(X_{\kappa}|_R, \partial X_{\kappa}|_{\partial R}\right)$, for $R$ a (co-oriented) minimal genus Seifert surface for $\kappa$, or more generally of the form $(M|_R, \partial M|_{\partial R})$, where $\partial M$ is a torus and $R$ is a properly embedded, co-oriented, compact, Thurston norm minimizing surface with a single boundary component.  In the latter cases we will denote $R_+(\gamma)$ simply by $R_+$ and, similarly, $R_-(\gamma)$ simply by $R_-$.

\begin{definition}[Definition 3.1 of \cite{Gab1},\cite{CCGabai}] 
Let $(M,\gamma)$ be a sutured manifold and let $S\subset M$  be a compact, properly imbedded, transversely oriented surface. Suppose  that one of the following holds for each component $\lambda$ of $S\cap\gamma$:
\begin{enumerate}
\item $\lambda$ is a properly imbedded, nonseparating arc in $A(\gamma)$.
\item $\lambda$ is a simple closed curve in a component $A$ of $A(\gamma)$ parallel to and coherently oriented with each component of $\partial A$.
\item If $T$ is a component of $T(\gamma)$ and $\lambda\subset T$, then $\lambda$ is an essential circle in $T$ and every component of $S\cap T$ is parallel to and coherently oriented with $\lambda$.
\end{enumerate}
Further suppose that, if $\alpha$ is a circle component of $S\cap R(\gamma)$, $\alpha$  does not bound a disk in $R(\gamma)$ nor is $\alpha$ the boundary of a disk component $D$ of $S$.  Finally, if the component  $\alpha$ of $S\cap R(\gamma)$ is a properly imbedded arc, suppose that it is not boundary compressible in $R(\gamma)$.  Then $S$ is called a \emph{decomposing surface} for $(M,\gamma)$.

A decomposing  surface $S$ defines a {\it sutured manifold decomposition} $$(M,\gamma)\overset{S}{\rightsquigarrow} (M',\gamma')$$ where
$$M'=M|_S$$ and 
\begin{eqnarray*}
\gamma' & = & (\gamma\cap M')\cup N(S'_+\cap R_-(\gamma))\cup N(S'_-\cap R_+(\gamma)),\\
R'_+(\gamma') & = & ((R_+(\gamma)\cap M')\cup S'_+)|_{A(\gamma)},  \mbox{    and} \\
R'_-(\gamma') & = & ((R_-(\gamma)\cap M')\cup S'_-)|_{A(\gamma)}.
\end{eqnarray*}
where $S'_+$ ($S'_-$) is that component of $\partial N(S)\cap M'$ whose normal vector points out of (into) $M'$.
This sutured manifold decomposition is called \emph{taut} if both $(M,\gamma)$ and $(M',\gamma')$ are taut.  
\end{definition}

\begin{definition}[Definition~3.4 of \cite{Gab1}] A sutured manifold $(M_0,\gamma_0)$ is \emph{decomposable} if there is a sequence of taut sutured manifold decompositions  $$\left(M_0,\gamma_0\right)\overset{S_1}{\rightsquigarrow}\left(M_1, \gamma_1\right)\overset{S_2}{\rightsquigarrow} \cdots \overset{S_n}{\rightsquigarrow}\left(M_n,\gamma_n\right)$$  such that $(M_n,\gamma_n)$ is a product sutured manifold. If  each component of $(M_n,\gamma_n)$ is homeomorphic to $(D^2\times [0,1],\partial  D_2\times [0,1])$, then the sequence is called a \emph{taut sutured manifold hierarchy}.
\end{definition}

\begin{notation}[Construction~4.16 of \cite{Gab3}] \label{constr4.16}
Suppose $n\ge 1$ and
$$\left(M_0,\gamma_0\right)\overset{S_1}{\rightsquigarrow}\left(M_1, \gamma_1\right)\overset{S_2}{\rightsquigarrow} \cdots \overset{S_n}{\rightsquigarrow}\left(M_n,\gamma_n\right)$$ is a sequence $\mathcal S$ of sutured manifold decompositions. Let $B^G(\mathcal S) =\langle S_1,\cdots, S_n\rangle$ denote the  co-oriented branched surface with underlying co-oriented spine $S_1\cup\cdots\cup S_n$.
\end{notation}

These branched surfaces $B^G(\mathcal S)$ will play a key role in the constructions found in this paper. The superscript G is used both to remind the reader that this is the branched surface described by Gabai and to distinguish this family of branched surfaces from the branched surfaces $B(\mathcal S')$ described in Section~\ref{modifications}. 

Note that, by definition of decomposing surface, $B^G(\mathcal{S})$ has no trivial bubbles. It follows immediately from Theorem~\ref{gabaithm1} below that $B^G(\mathcal{S})$ is taut if   $\mathcal{S}$ is a sequence of taut sutured manifold decompositions.

We will use extensively the following result of Gabai:

\begin{thm}[Theorems~3.13, 4.2 and 5.1 of \cite{Gab1}]   \label{gabaithm1}
Suppose $(M,\gamma)$ is taut. Then $(M,\gamma)$ is decomposable, with a choice $\mathcal S$  of taut sutured manifold   hierarchy  such that Construction~4.17 of \cite{Gab3}  applied to $B^G(\mathcal S)$   yields a   co-oriented taut foliation fully carried by $B^G(\mathcal S)$.
\end{thm}

 We are primarily interested in the case that $(M,\gamma)$ is the complement of an I-fibered neighbourhood of   a taut co-oriented branched surface $B$  that fully carries a lamination $\mathcal{L}$. In this setting, we are interested in applying Theorem~\ref{gabaithm1} to extend $\mathcal{L}$ to a co-oriented taut foliation.   Construction~4.17 of \cite{Gab3} guarantees that this is always possible. This follows from the fact that there is a branched surface $B'$, obtained from $B$ by splitting along at most finitely many compact surfaces of contact, such that $B'$ fully carries $\mathcal{L}$ and  $(B',\mathcal{L})$ satisfies the following \emph{noncompact extension property}.

\begin{definition} [Noncompact extension property; see \cite{DR2} for details] Suppose $B$ fully carries a lamination $\mathcal{L}$. A component $A$ of $\partial_v N(B)$ satisfies the \textit{noncompact extension property} relative to the pair $(B,\mathcal{L})$ if there is a copy of $[0,\infty)\times [0,1]$ properly embedded in $N(B) \cap M|_\mathcal{L}$ with  $\{0\}\times [0,1]$ contained in an I-fiber of $N(B)$ and containing an I-fiber of $A$, and  $[0,\infty)\times \{0,1\}$  contained in leaves of $\mathcal{L}$.

$(B,\mathcal{L})$ satisfies the \textit{noncompact extension property} if each component of $\partial_v N(B)$ does.  
\end{definition}

\begin{cor}[\cite{DR2}] \label{cor:canextend} 
Suppose $B$ is a co-oriented branched surface that fully carries a lamination $\mathcal{L}$. If $B$ is taut, then there is a co-oriented taut foliation that contains $\mathcal{L}$ as a sublamination.   Indeed, if 
$\mathcal S$ is  the taut sutured manifold hierarchy
$$\left(M \setminus \Int N(B), \gamma\right)\overset{S_1}{\rightsquigarrow}\left(M_1, \gamma_1\right)\overset{S_2}{\rightsquigarrow} \left(M_2, \gamma_2\right)\overset{S_3}{\rightsquigarrow}\cdots \overset{S_n}{\rightsquigarrow}\left(M_n,\gamma_n\right),$$
where $\gamma = \partial_v N(B) \cup (\partial M\setminus \text{int}N(B))$,
 then the branched surface obtained by applying Gabai's Construction~4.16 to $\mathcal S$, starting with $B$, yields a taut branched surface that fully carries a taut foliation containing $\mathcal{L}$ as a sublamination.

\end{cor}

The following proposition shows that if the complementary region containing a torus boundary component of $M$ is of a particular type and also has vertical boundary whose components all satisfy the noncompact extension property, then $\mathcal{L}$ has extensions to taut, co-oriented foliations realizing every slope, except one,  on that boundary component.

\begin{prop} [\cite{DR2}] \label{prop:boundaryslopes} Let $\partial_0 M$ denote a torus boundary component of $M$. 
Suppose $B$ is a taut co-oriented branched surface that fully carries a lamination $\mathcal{L}$ and is disjoint from $\partial_0 M$. If the complementary region   $(Y,\partial_v Y)$ of $N(B)$ containing $\partial_0 M$ is homeomorphic to  $$\left(\partial_0 M \times [0,1], V_1 \cup\cdots\cup V_{2n} \right),$$ where $\partial_0 M\times \{0\}=\partial M$ and $V_1,\cdots,V_{2n}$ are disjoint essential annuli in $\partial_0 M\times \{1\}$ that satisfy the noncompact extension property relative to $(B,\mathcal{L})$, then there are co-oriented taut foliations  extending $\mathcal{L}$ that strongly realize all boundary slopes   on $\partial_0M$ except the one isotopic in   $Y$ to the  core of any  $V_i$. 
\end{prop}

\section{Modifying sutured manifold decompositions} \label{modifications}

Some sutured manifold hierachies are better than others. In this section, we focus on the case that there is a sutured manifold hierarchy 
$$\left(M_0,\partial M_0\right)\overset{R}{\rightsquigarrow}\left(M_1, \gamma_1\right)\overset{S}{\rightsquigarrow} \left(M_2, \gamma_2\right)\overset{S_3}{\rightsquigarrow}\cdots \overset{S_n}{\rightsquigarrow}\left(M_n,\gamma_n\right)$$
  in which  $R$ is a minimal Seifert surface, and  $S$ is \emph{double-diamond taut} (see Definition~\ref{defn:ddtaut}). In this case, it is possible to construct an associated branched surface that fully carries co-oriented taut foliations that strongly realize all boundary slopes except one. For simplicity of exposition, we assume in what follows that $\partial M_0$ is connected. 
  
  Roughly speaking, the branched surface $B^G(\mathcal S)$ of Gabai is modified as follows. Before beginning the hierarchy, a boundary-parallel torus, which we denote by $T$, is added, separating $M$ into two components. The component homeomorphic to $M_0$ then replaces $M_0$ in the sutured manifold hierarchy above. The union of surfaces $T\cup R\cup S\cdots\cup S_n$ is smoothed to a branched surface $B(\mathcal S')$  in two steps. First, the forthcoming  special property of $S$ (the \emph{double-diamond} condition) is used to smooth $T\cup R\cup S$, in such a way that the resulting branched surface is laminar and  has taut sutured manifold complement. This complement consists of two components: one is isomorphic to the complement of  the branched surface $\langle R,S\rangle$ and the other is isomorphic to  $(T^2\times [0,1],(V_1\cup V_2)\times \{0\})$, where $T^2\times \{0\}=T$, $T^2\times \{1\}=\partial M_0$, and $V_1$ and $V_2$ are disjoint annuli with slope a meridian $\nu$, which in the case of a knot complement is necessarily the usual meridian. In the second step, the addition of the surfaces $S_1,\cdots, S_n$ and the smoothing to $B(\mathcal S')$ proceed as described in Gabai's Construction~4.16; see  Notation~\ref{constr4.16}.    
  
The branched surface $B(\mathcal S')$ fully carries a lamination $\mathcal L$  and has complement the  disjoint union of $\left(M_n,\gamma_n\right)$ with $(T^2\times [0,1],(V_1\cup V_2)\times \{0\})$.  It follows that $\mathcal L$ admits extensions to a family of foliations strongly realizing all boundary slopes except $\nu$.

\begin{notation} \label{Notation4.1} Let $M=M_0$ be a  compact  oriented 3-manifold with $\partial M$ a torus. Let    $T$ be a boundary parallel torus properly embedded   in $M$, and let $X$ and $Y$ denote the components of $M|_T$, with $X$ the component homeomorphic to $M$.  Associated to a sequence $\mathcal S$ of  sutured manifold decompositions
$$\left(M_0,\gamma_0\right)\overset{S_1}{\rightsquigarrow}\left(M_1, \gamma_1\right)\overset{S_2}{\rightsquigarrow} \cdots \overset{S_n}{\rightsquigarrow}\left(M_n,\gamma_n\right),$$
(where we assume that $S_i \cap Y = \partial S_i \times I$, for each $i$) is the isomorphic sequence $\mathcal S'$
$$\left(X,T\right)\overset{S'_1}{\rightsquigarrow}\left(X_1, \gamma'_1\right)\overset{S'_2}{\rightsquigarrow} \cdots \overset{S'_n}{\rightsquigarrow}\left(X_n,\gamma'_n\right),$$
where $S_i'=S_i\cap X$. Let $\Sigma(\mathcal S)$ denote the co-oriented spine $S_1\cup\cdots \cup S_n$, and let $\Sigma(\mathcal S')$ denote the \underline{unoriented} spine $T\cup S_1'\cup\cdots \cup S_n'$.
\end{notation}

\begin{definition}
Two 1-manifolds properly embedded in a common surface intersect \emph{efficiently} if any intersections are transverse and no isotopy through properly embedded 1-manifolds reduces the number of points of intersection.
\end{definition}

\begin{definition}\label{defn:tight} 
Let $M$ be an oriented 3-manifold with connected boundary $\partial M$ a torus, and suppose $M$ contains a properly embedded, compact, oriented,  Thurston norm minimizing surface $R$ with a single boundary component. Call a decomposing surface $S$ for $\left(M|_R, \partial M|_{\partial R}\right)$  \textit{tight} if
\begin{enumerate}
\item the images of the 1-manifolds $\alpha_+=S\cap R_+$ and $\alpha_-=S\cap R_-$ under the quotient map $M|_R \to M$ intersect efficiently in $R$, and
\item the sutured manifold decomposition $\left(M|_R, \partial M|_{\partial R}\right)\overset{S}{\rightsquigarrow} (M',\gamma')$ 
  is  taut.
\end{enumerate}

\end{definition}

\begin{definition}[\cite{Gabfib}]  \label{proddisk}  A  decomposing surface $D$ is called a \emph{product} disk if   $D$ is a disk and $|D\cap A(\gamma)|=2$.
\end{definition}

\begin{notation}  Let  $S$ be a tight decomposing surface for $\left(M|_R, \partial M|_{\partial R}\right)$.  Let $\partial_1 S,\cdots,\partial_s S$ be the components of $\partial S$ that have nonempty intersection with $A(\gamma)$.  Each such component can be written as a concatenation of arcs that lie alternately in   $R_+$, $\partial M$, and $R_-$. Letting $*$ denote concatenation, and denoting the (even) number of arcs of $\partial_i S$ on  $\partial M$ by $2n_i$, we write:
  $$\partial_i S=\alpha_{i,1}*\tau_{i,1}*\cdots\alpha_{i,2n_i}*\tau_{i,2n_i},\,\,\,1\le i \le s,$$ where $\alpha_{i,2j-1} \subset R_+$, $\alpha_{i,2j} \subset R_-$, and  $\tau_{i,k} \subset \partial M$, for $j=1, 2, \ldots n_i$ and $k= 1, 2, \ldots, 2n_i$.

\end{notation}

\begin{definition}
We call the arcs $\tau_{i,k}$, as well as the corresponding arcs in the spine $\Sigma(\mathcal{S}')$,  \emph{transition arcs}. 
\end{definition}

\begin{definition}
Given a framing on $\partial M$ with longitude given by the oriented curve $\partial R$ and arbitrary fixed meridian $\mu_0$, we may talk of the \emph{sign} of the transition $\tau_{i,k}$ relative to $\mu_0$. This sign is defined as follows. Orient $\tau_{i,k}$ so that it is a path from $R_+$ to $R_-$. Associated to the framing $\partial M=S^1\times S^1$,  with $S^1\times \{1\}$ representing the preferred longitude and $\{1\}\times S^1$ representing the preferred meridian, is a retraction   $r_{\mu_0}:S^1\times S^1\to S^1$ given by   $r_{\mu_0}(z,w)=z$. Define a transition arc to be \emph {positive (negative) rel $\mu_0$} if $\tau_{i,k}$ can be isotoped in $\partial M$ rel endpoints so that the restriction of   $r_{\mu_0}$ to $\tau_{i,k}$ is orientation preserving (reversing). If the meridian $\mu_0$ is understood, refer to   the associated retraction as $r$ and to a transition arc as being positive or negative. This designation of positive or negative is independent of the choice of orientation chosen on $R$.

For an alternative viewpoint that is often useful, observe that associated to a transition arc  $\tau_{i,k}$ (oriented as above), there are exactly two choices of meridian with representatives disjoint from $\tau_{i,k}$. These are obtained from the two simple closed curves obtained by pasting $\tau_{i,k}$ with one of the two components of   $\partial R|_{\partial \tau_{i,k}}$.  For either such choice of meridian, $\upsilon = r(\tau_{i,k})$ is a proper arc of $S^1 \times \{1\}$, and $\tau_{i,k}$ is positive or negative, respectively, if $\upsilon$ inherits positive or negative orientation along $\partial R$ from $\tau_{i,k}$, as illustrated in Figure \ref{transsign}.  The transition arc is positive with respect to one choice and negative with respect to the other.

\end{definition}

\begin{figure}[ht]
\scalebox{.25}{\includegraphics{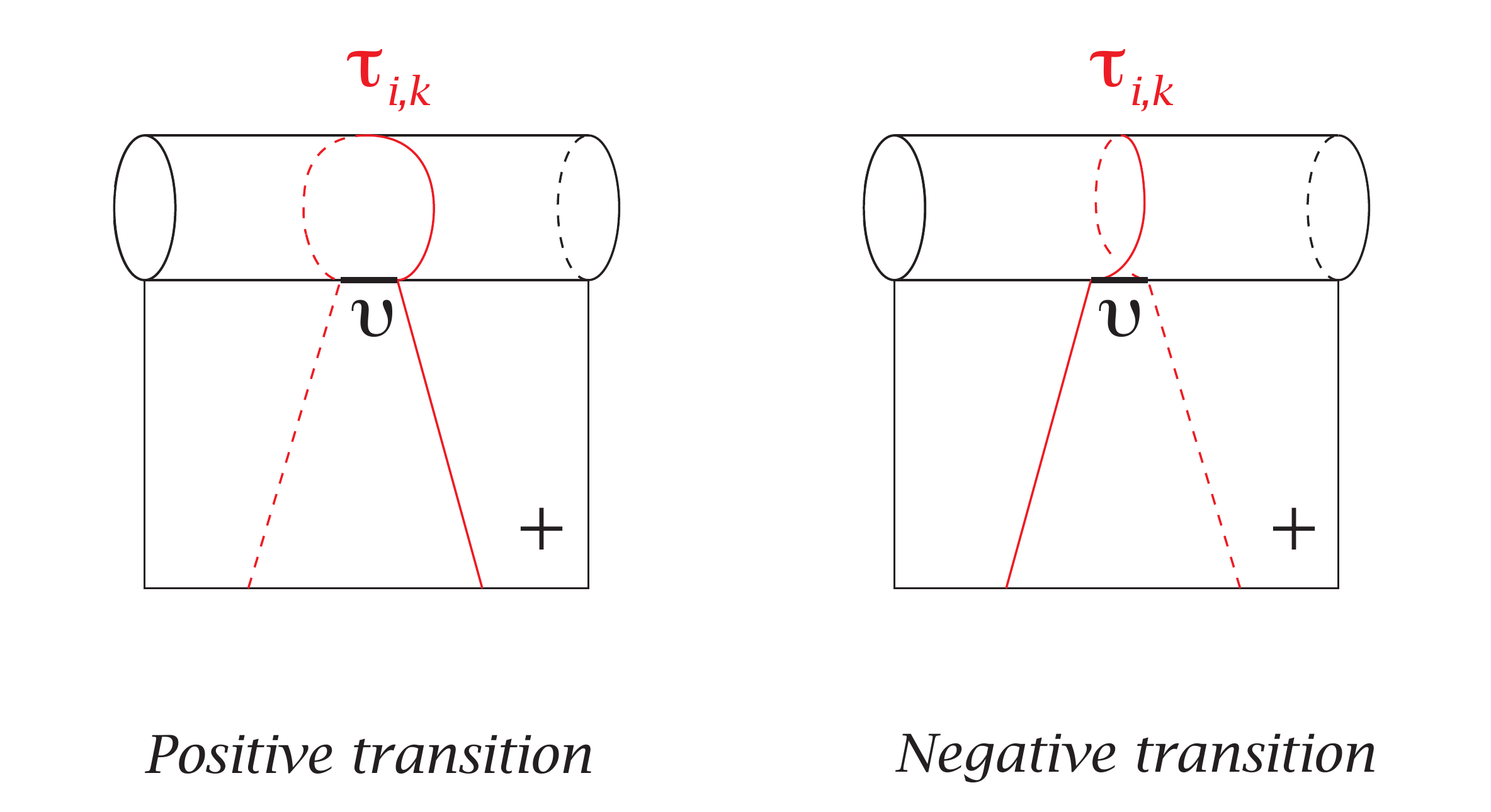}}
\caption{Positive and  negative  transition models.  }\label{transsign}
\end{figure}

\begin{definition}\label{defn:ddtaut}  Suppose that $S$ is a tight decomposing surface for $\left(M|_R, \partial M|_{\partial R}\right)$ that is not a product disk.    For simplicity of notation, identify arcs in   $R_+ \cup R_-$ with their images in $R$ under the quotient map $M|_R \to M$.  Suppose that $s\ge 1$ and there exist $i,j$ such that $\alpha_{i,j}$ and $\alpha_{i,j+1}$ are isotopic   through proper embeddings in $R$,  and let  $\Delta$ be the component of $R|_{\alpha_{i,j}\cup\alpha_{i,j+1}}$ that contains $\tau_{i,j}$. If $\Delta$ is a source  when considered as a sector of the branched surface $\langle R,S\rangle$, then call $S$ and the sutured manifold decomposition $\left(M|_R, \partial M|_{\partial R}\right)\overset{S}{\rightsquigarrow} (M',\gamma')$ \emph{double-diamond taut} (with respect to $\alpha_{i,j}*\tau_{i,j}*\alpha_{i,j+1}$).  
\end{definition}

In particular, a double-diamond taut decomposition is taut.  In addition,  a double-diamond taut decomposition determines a \emph{distinguished meridian}, which we denote by $\nu$, defined as follows.

\begin{definition} Suppose $S$ is double-diamond taut with respect to $\alpha_{i,j}*\tau_{i,j}*\alpha_{i,j+1}$.  Define the distinguished meridian $\nu$ on $\partial M$ to be the isotopy class of  
$$\nu_{i,j} = \tau_{i,j}\cup \upsilon_{i,j},$$ 
where $\upsilon_{i,j}$ in the  unique arc  in $\partial R  \cap \Delta$ joining the endpoints of $\tau_{i,j}$. 
\end{definition}
  Note that $\nu_{i,j}$ is  a simple closed curve that  has geometric intersection number one with $\partial R$ and hence  is a meridian.  Since $S$ is tight and $\alpha_{i,j}$ and $\alpha_{i,j+1}$ are isotopic, the transition arcs $\tau_{i,j-1}$, $\tau_{i,j}$, and $\tau_{i,j+1}$ have the same sign rel $\nu$. For each transition arc $\tau_{i,k}$, let $\nu_{i,k}$ denote the representative $\tau_{i,k} \cup r(\tau_{i,k})$ of $\nu$.  It will follow from Corollary~\ref{main theorem} that if  $M = X_\kappa$ for a knot $\kappa$, then $\nu = \mu$, the meridian of $\kappa$

  In the following theorem statement and proof, we use  Notation~\ref{Notation4.1}. In particular, $T$ is a boundary parallel torus properly embedded in $M$.    We also refer to the curve on $T$ isotopic to $\nu$ as the distinguished meridian and use the same notation for it. 

\begin{theorem}[Double-diamond replacement] \label{thm:Btodoublediamond}
  Let $\mathcal S_0$ be a sequence of taut sutured manifold decompositions $$\left(M,\partial M\right)\overset{R}{\rightsquigarrow}\left(M_1, \gamma_1\right)\overset{S}{\rightsquigarrow} \left(M_2, \gamma_2\right)$$  such that $\partial M$ is a torus, $\partial R$ is connected and nonempty, and  $S$ is  double-diamond taut with respect to $$\alpha_{i,j}*\tau_{i,j}*\alpha_{i,j+1}\subset \partial S.$$

Let $\Delta$ be the component of $R|_{\alpha_{i,j}\cup\alpha_{i,j+1}}$ that contains $\tau_{i,j}$, and suppose    
 each sector of  $\Sigma(\mathcal S_0')\cap \Sigma(\mathcal S_0)$  except $\Delta$ inherits its co-orientation from that of $\Sigma(\mathcal S_0)$. Give $\Delta$ the opposite orientation from $R$. There is a choice of orientations on the  sectors   of   $T$ in $\Sigma(\mathcal S_0')$ so that $\Sigma(\mathcal S_0')$ admits a smoothing to a (co-oriented) taut branched surface $B(\mathcal S_0')$   that has complement   comprised of two sutured manifold  components: one is $B^G(S_0)$ and the other is
   $\left(\partial M \times [0,1], V_1 \cup V_2 \right)$, where $\partial M\times \{0\}=\partial M$ and $V_1$ and $V_2$ are disjoint annuli of slope $\nu$, the distinguished meridian   determined by $S$.  

\end{theorem}

\begin{proof}  In order to simplify diagrams, we enhance arrow notation for cusp directions with diamond notation, as originally introduced by Wu \cite{Wu} and depicted in Figure \ref{diamond}.  

\begin{figure}[ht]
\scalebox{.3}{\includegraphics{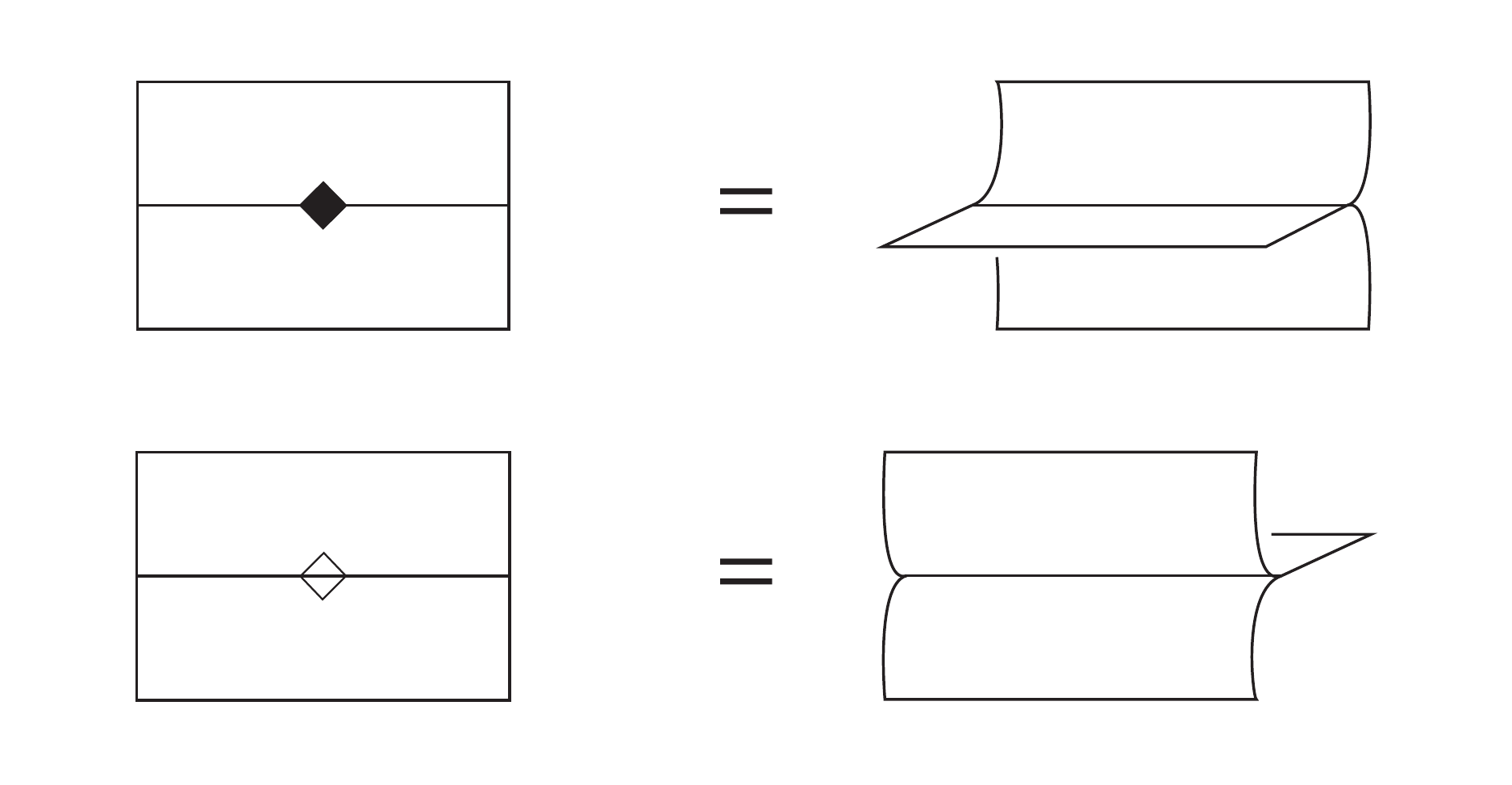}}
\caption{Diamond notation}  \label{diamond}
\end{figure}

 The desired co-orientations on the sectors of $T$, along with the resulting smoothing and the paths of the sutures that result, are indicated for negative transitions in Figure \ref{double-diamond replacement};  the construction for positive transitions is given by the mirror image. A schematic treatment for negative transitions is provided in Figure \ref{local models and schematic}, and for positive transitions in Figure \ref{positive schematic}.  (The reason for the mysterious term ``double-diamond" should now be clear!)  The local models  for co-orientations and associated smoothings in neighborhoods of $\nu_{i,j-1}, \nu_{i,j}$ and $\nu_{i,j+1}$ are highlighted in Figures \ref{local models and schematic} and \ref{positive schematic}.    It is straightforward to check that the resulting co-orientation on  $\Sigma(\mathcal S_0')$ admits a smoothing to a co-oriented branched surface,  which we denote by $B(\mathcal{S}_0')$. 
  
The local smoothings at $\nu_{i,j-1}$ and  $\nu_{i,j}$ (marked $C$ in  Figures~\ref{local models and schematic}, \ref{positive schematic} and \ref{suture comparison}) each introduce a cusp of slope $\nu$ on $T$ with sink direction pointing into $R$   and $S$, resulting in a suture on the boundary of the complementary region $Y$  containing  $\partial M$.  As shown in Figures \ref{double-diamond replacement} and \ref{suture comparison},   the remaining complementary  region is  isotopic to   the complement of $B^G(\mathcal{S}_0)$.  The complementary regions of $B(\mathcal{S}_0')$ are therefore taut sutured manifolds.   

It remains only to show that   through every sector of $B(\mathcal{S}_0')$ there is a closed  oriented curve that is positively transverse to $B(\mathcal{S}_0')$. For points in $\Delta$, such a transversal is provided by a curve $\alpha$ parallel to $\nu$, appropriately oriented.  It follows from Theorem \ref{gabaithm1} and Definition \ref{defn:tight} that $B^G(\mathcal{S}_0)$ is taut;  therefore, through any point in a sector disjoint from $T$ and $\Delta$ there is a closed positive transversal to $B^G(\mathcal{S}_0)$.  It is possible that this transversal passes through $\Delta$, in which case such intersections may be removed by inserting copies of $\alpha$.   Finally, let $\Gamma$ be any sector of $B(\mathcal{S}_0')$ that is contained in $T$. Without loss of generality, we may assume that a positive transversal to $\Gamma$ points into a complementary region disjoint from $\partial M$.  Let $\beta_1$ be an arc properly embedded in the complement of $B(\mathcal{S}_0')$ that runs from $\Gamma$ to $S_-$ (viewed as pushed slightly off the negative side of $S$), let $\beta_2$ be a simple arc in $S_-$ from the endpoint of $\beta_1$ to a point in a sector of $T$ adjoining a meridional cusp, and let $\beta_3$ be a simple arc properly embedded in the complementary region containing $\partial M$ that connects $\beta_3$ to $\beta_1$. The simple closed curve $\beta = \beta_1*\beta_2*\beta_3$ is the desired transversal.  (See Figure \ref{transversals}.)
\end{proof}

\begin{figure}[ht]
\scalebox{.5}{\includegraphics{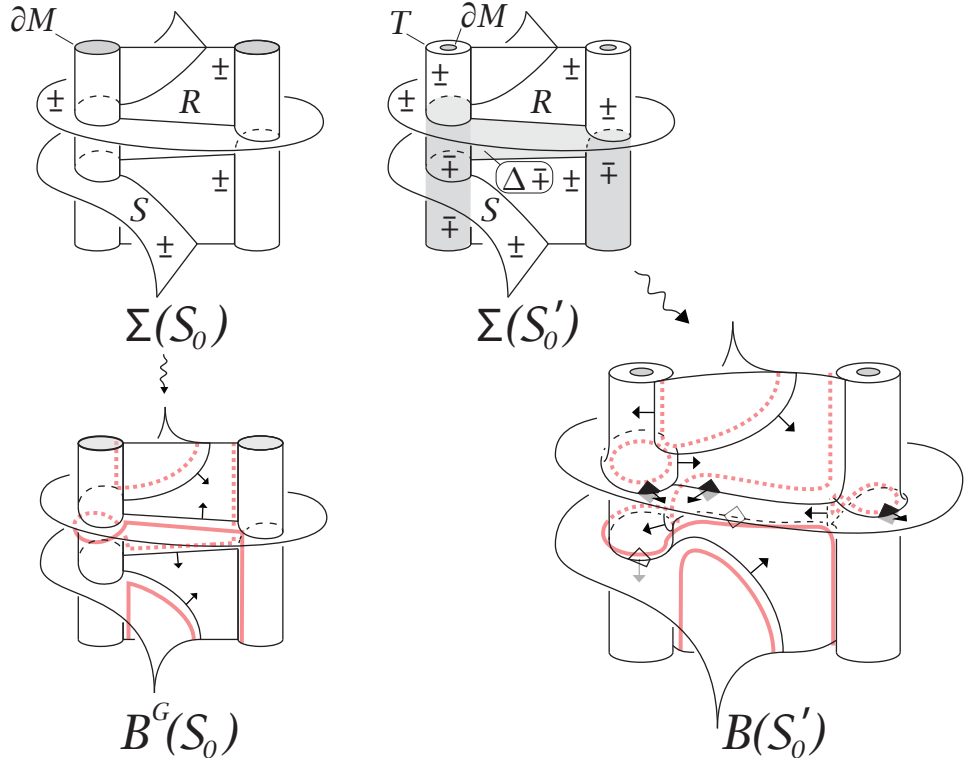}}
\caption{The transverse orientations on $\Sigma(\mathcal{S}_0)$ and $\Sigma(\mathcal{S}_0')$ and the resulting branched surfaces $B^G(\mathcal{S}_0)$ and $B(\mathcal{S}_0')$, with the paths of sutures in red.  Sutures behind surface $R$, as well as the meridional sutures in component $Y$, are dashed;  those in front of $R$ are solid.}\label{double-diamond replacement}
\end{figure} 

\begin{figure}[ht]
\scalebox{.45}{\includegraphics{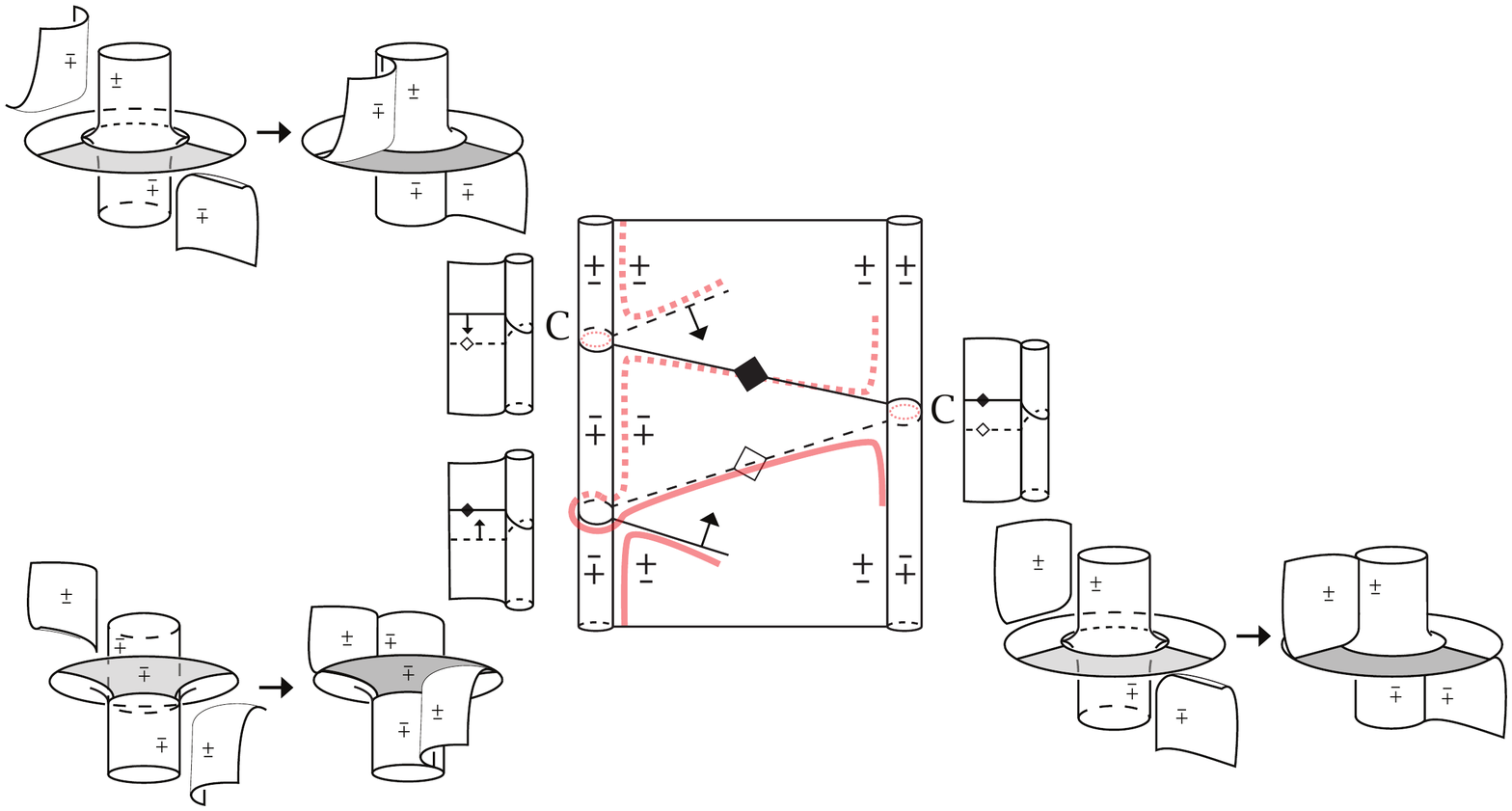}}
\caption{The local models in neighborhoods of   $\nu_{i,j-1}, \nu_{i,j}$ and $\nu_{i,j+1}$.}\label{local models and schematic}
\end{figure} 

\begin{figure}[ht]
\scalebox{.45}{\includegraphics{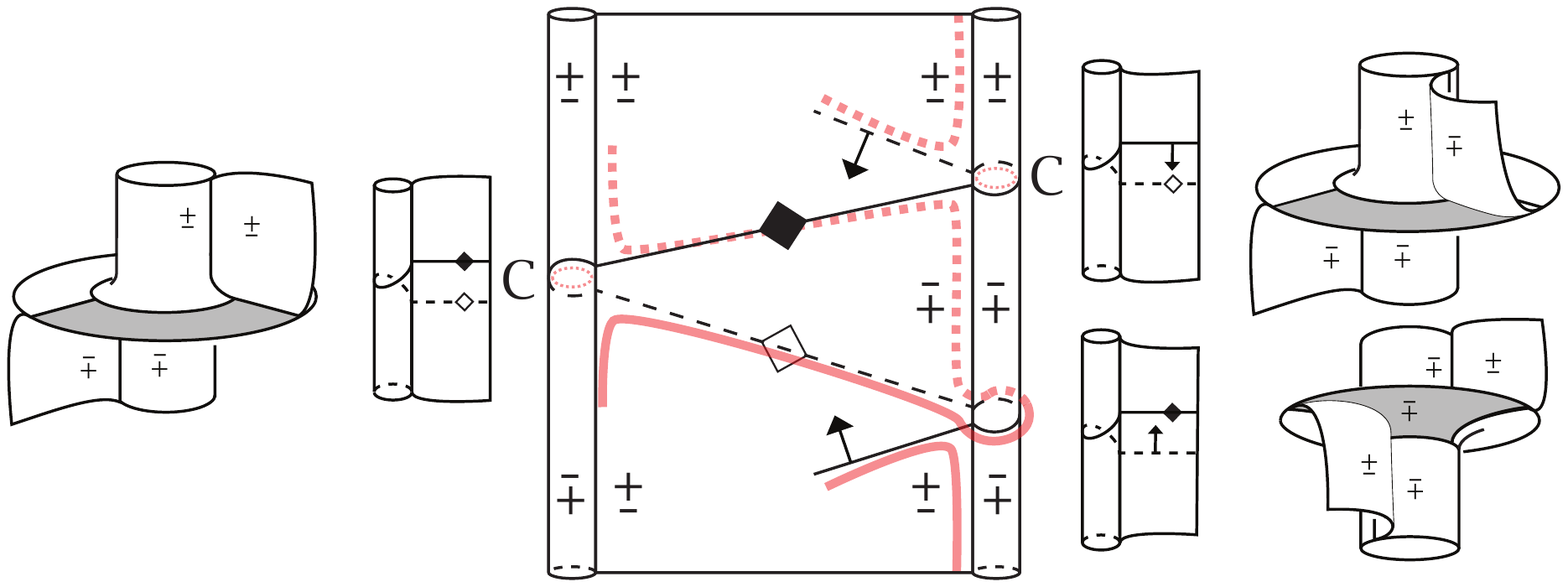}}
\caption{The local models in neighborhoods of   $\nu_{i,j-1}, \nu_{i,j}$ and $\nu_{i,j+1}$.}\label{positive schematic}
\end{figure} 

\begin{figure}[ht]
\scalebox{.4}{\includegraphics{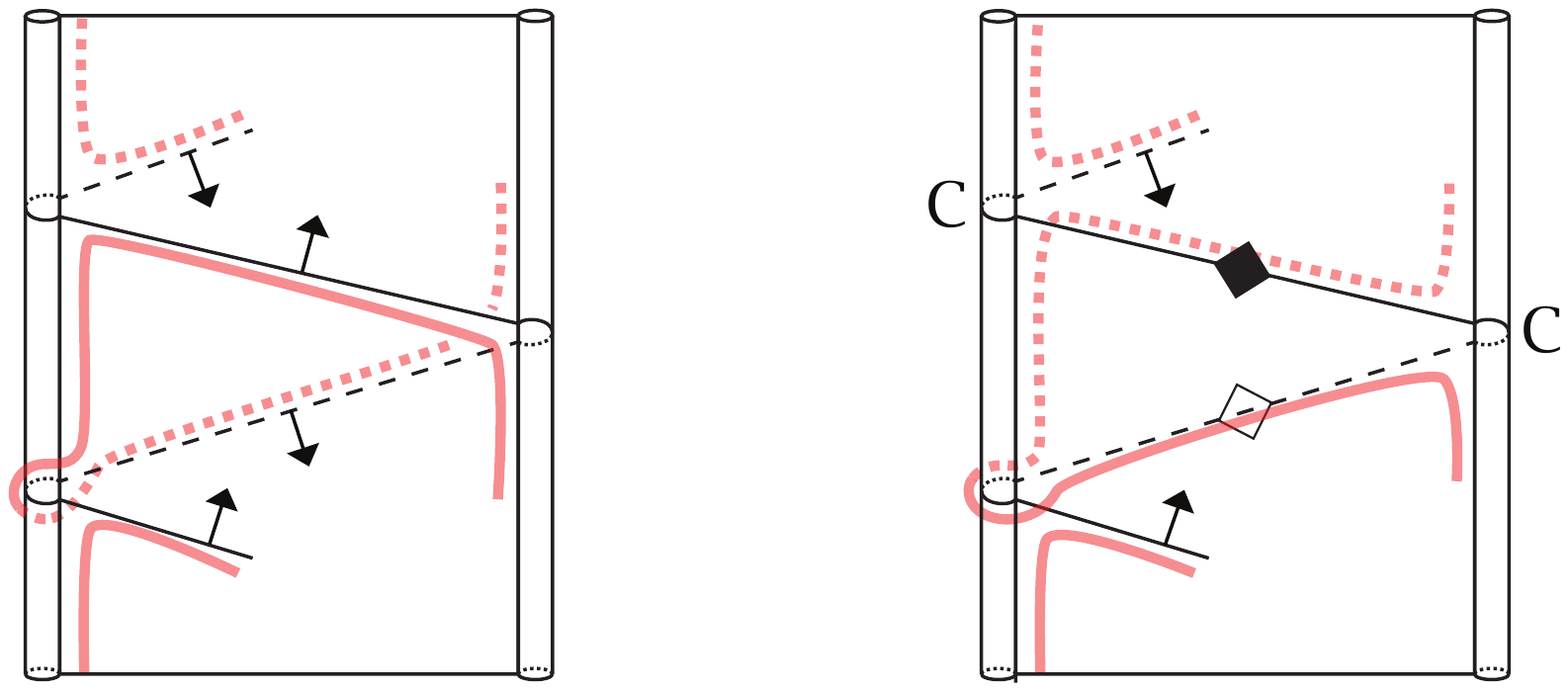}}
\caption{Comparison of the sutures in the complements of $B^G(\mathcal{S}_0)$ and $B(\mathcal{S}_0')$.}\label{suture comparison}
\end{figure} 

\begin{figure}[ht]
\scalebox{.3}{\includegraphics{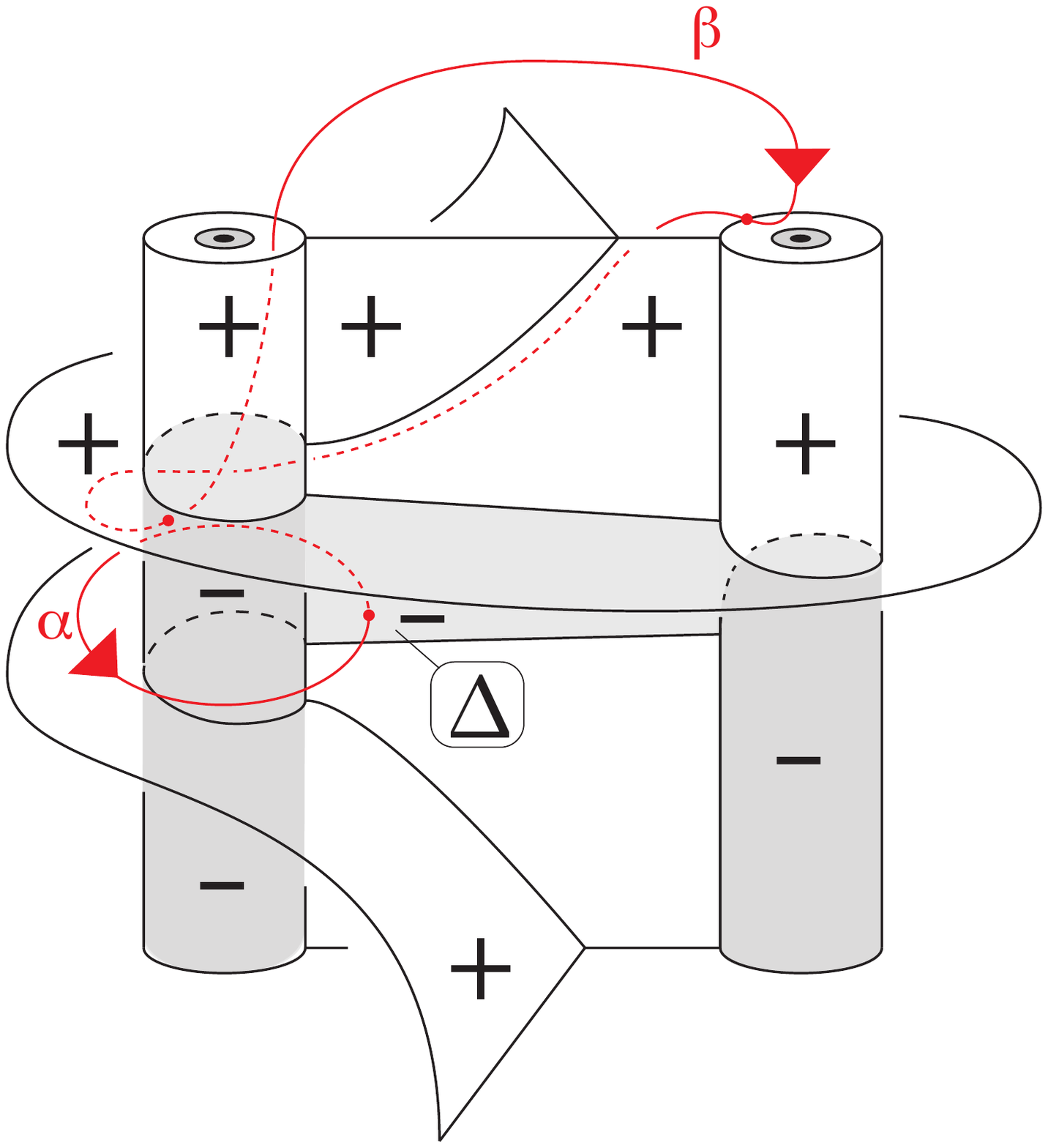}}
\caption{Positive closed transversals for $\Delta$ and the sectors of $T$ in $B(\mathcal{S}'$).}\label{transversals}
\end{figure}

We have now established the desired taut branched surface, but in contrast to Gabai's branched surface $B^G(\mathcal{S}_0)$, it is not inherent in the construction of $B(\mathcal{S}_0')$ that it fully carries a lamination.  However, an easily checked technical condition on $B^G(\mathcal{S}_0)$ suffices to ensure that it does. We then apply Corollary~\ref{cor:canextend} and Proposition~\ref{prop:boundaryslopes}  to show this lamination extends to taut foliations that strongly realize all boundary slopes except the distinguished meridian $\nu$.  

\begin{cor} \label{main theorem}
   Let $\mathcal S_0$ be a sequence of taut sutured manifold decompositions $$\left(M,\partial M\right)\overset{R}{\rightsquigarrow}\left(M_1, \gamma_1\right)\overset{S}{\rightsquigarrow} \left(M_2, \gamma_2\right)$$ such that $\partial M$ is a torus, 
   $\partial R$ is connected and nonempty, and  $S$ is  double-diamond taut with respect to $$\alpha_{i,j}*\tau_{i,j}*\alpha_{i,j+1}\subset \partial S.$$

 If $B^G(\mathcal{S}_0)$ has no sink disks disjoint from $\partial M$, then there are co-oriented taut foliations   that strongly realize all boundary slopes except $\nu$, the distinguished meridian determined by $S$.  Moreover, if $M=X_{\kappa}$  for some knot $\kappa$ in an L-space, $\kappa$  is persistently foliar.  

\end{cor}

\begin{proof}   Let $B(\mathcal S_0')$ be the   taut, and hence essential, branched surface obtained from $B^G(\mathcal S_0)$ by double-diamond replacement, as described in Theorem~\ref{thm:Btodoublediamond}. Since $B^G(\mathcal S_0)$ contains no trivial bubbles, neither does $B(\mathcal S_0')$.   We  show that $B(\mathcal S_0')$ contains no sink disks, and hence is laminar.

 The sectors of $B(\mathcal S_0')$  that are contained in $T$  correspond to digon complementary regions of $B^G(\mathcal S_0)\cap \partial M$ and hence are not sink disks. We  therefore    restrict attention to
the sectors of $B(\mathcal S_0')$ that are not contained in $T$.  These  are in bijective correspondence with, and isomorphic to, those in $B^G(\mathcal S_0)$, with  arrows replaced by diamonds   at $\alpha_{i,j}\cup\alpha_{i,j+1}$. In addition, outward pointing arrows are  introduced along boundary arcs   on $T$ except for those adjacent to one of the two   meridian cusps, marked $C$ in Figures~\ref{local models and schematic} and \ref{suture comparison}.   Any sector with boundary  that has nonempty intersection with $T$ therefore has outward pointing arrow along $T$; in particular, this is true for the sectors that have nonempty intersection with $\alpha_{i,j}\cup\alpha_{i,j+1}$. Moreover, by assumption, any sector that has boundary disjoint from $T$ is not a sink disk.    Thus, $B(\mathcal S_0')$ is laminar, and therefore fully carries a lamination  $\mathcal{L}$.

  Since any leaf of $\mathcal{L}$ that passes through the I-fibered neighbourhood of the  branch $S$ is asymptotic to $T$, it is noncompact. Hence the conditions of Proposition~\ref{prop:boundaryslopes} are satisfied. It  therefore follows from Corollary~\ref{cor:canextend} and Proposition~\ref{prop:boundaryslopes}   that $M$ supports co-oriented taut foliations that strongly realize all boundary slopes except $\nu$. In the case $M = X_\kappa$ for a knot $\kappa$ in an L-space, it must be the case that $\nu = \mu$, the meridian of $\kappa$, since by   \cite{OzSz,Bowden,KR1,KR2} an L-space supports no co-oriented taut foliation;  hence $\kappa$ is persistently foliar. 

\end{proof}

\begin{remark}
If $S$ is a tight product disk and $\alpha_1 = \alpha_{1,1}$ and $\alpha_2 = \alpha_{1,2}$ are isotopic  through proper embeddings in $R$, then after isotoping $D$ so that $\alpha_1 = \alpha_2$, $D$ determines a meridional annulus, and hence $\kappa$ is composite,   as illustrated in Figure \ref{composite}.   For some results on persistently foliar composite knots, see \cite{DR1}.
\end{remark}

\begin{figure} 
\scalebox{.4}{\includegraphics{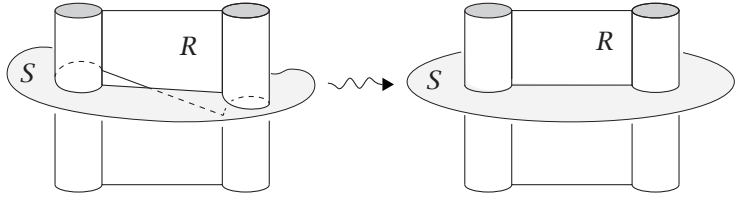}}
\caption{If $S$ is a tight product disk and $\alpha_1 = \alpha_{1,1}$ and $\alpha_2 = \alpha_{1,2}$ are isotopic  through proper embeddings in $R$, then $\kappa$ is composite.} \label{composite}
\end{figure}

\begin{remark}
Definition~\ref{defn:tight} admits a straightforward generalization to sutured manifold decompositions of  greater length; namely, to appropriately constrained decompositions of the form
$$\left(\left(M, \partial M\right)\overset{S_1}{\rightsquigarrow}\left(M_1, \gamma_1\right)\overset{S_2}{\rightsquigarrow} \cdots \overset{S_n}{\rightsquigarrow}\left(M_n,\gamma_n\right)\right)\overset{S}{\rightsquigarrow}\left(M',\gamma'\right).$$ This generalization then leads naturally to a generalization of Theorem~\ref{thm:Btodoublediamond}. As these generalizations require  significantly more cumbersome notation, and have yet to prove useful in explicit applications,    we leave their statements as   exercises for the interested reader. 
\end{remark}

\section{Application to Murasugi sums}\label{Murasugi sums}

\begin{definition}[\cite{GabMurasugi}] Let $n\ge 1$. The oriented surface $R$ in  $S^3$ is a \textit{Murasugi sum} along    $P$ of  oriented  surfaces $R_1$ and $R_2$   if 
\begin{enumerate}
\item  there is a sphere $S$ in $S^3$ such that $P=S\cap R$ is a $2n$-gon embedded in $R$ with boundary   a concatenation $\partial P=\alpha_1*\beta_1*\cdots \alpha_n*\beta_n$ such that each $\alpha_i$ is a subarc of $\partial R$ and each $\beta_i$ is properly embedded in $R$, and
\item if $B_1$ and $B_2$ are the two embedded 3-balls with common boundary $S$, then  for each $i$, 
$R_i$ is homeomorphic to the subsurface $B_i\cap R$. 
\end{enumerate}
\end{definition}

When $n=2$, a Murasugi sum is also known as a   \emph{plumbing} (see Figure \ref{band}).  For a nice history of the  definitions of  plumbing and Murasugi sum, see \cite{ozbagcipp}.

\begin{theorem}[\cite{GabMurasugi,GabMurasugi2}]\label{gabaimurasugi}  
Suppose $R\subset S^3$ is a Murasugi sum along $P$ of oriented surfaces $R_1$ and $R_2$. Then $R$ is a minimal genus Seifert surface for $\partial R$ if and only if each $R_i$, $i=1,2$, is a minimal genus Seifert surface for $\partial R_i$.
\end{theorem}

\begin{cor} \label{plumbing corollary}
Suppose $\kappa$ is a knot in $S^3$ with minimal genus Seifert surface $R$. 
If $R$ is a plumbing  of surfaces $R_1$ and $R_2$, where $R_2$ is an unknotted band with an even number $2m\ge 4$ of half-twists, then the decomposing disk for $R_2$ can be chosen to be  double-diamond taut. Thus, $\kappa$ is persistently foliar.
\end{cor}

\begin{proof}  The  decomposing disk $D$ dual to the band $R_2$ is double-diamond taut, as indicated in Figure~\ref{band}. Moreover,   $B^G(\mathcal{S}_0) = \langle R, D\rangle$  has no sink disks disjoint from $\partial X_\kappa$.  The result therefore follows immediately from Corollary~\ref{main theorem}.     The relevant parts of $B^G(\mathcal{S}_0)$ and $B(\mathcal{S}_0')$ are shown in Figure \ref{suture comparison 2}.

\end{proof}

\begin{figure}[ht]
\scalebox{.25}{\includegraphics{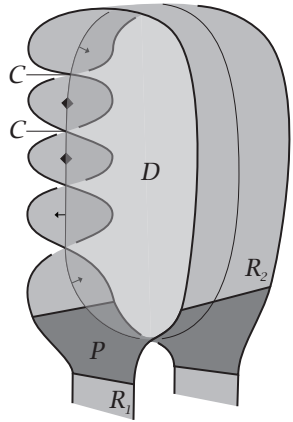}}
\caption{$B(\mathcal{S}_0')$ for a positive plumbed band.}\label{band}
\end{figure} 

\begin{figure}[ht]
\scalebox{.4}{\includegraphics{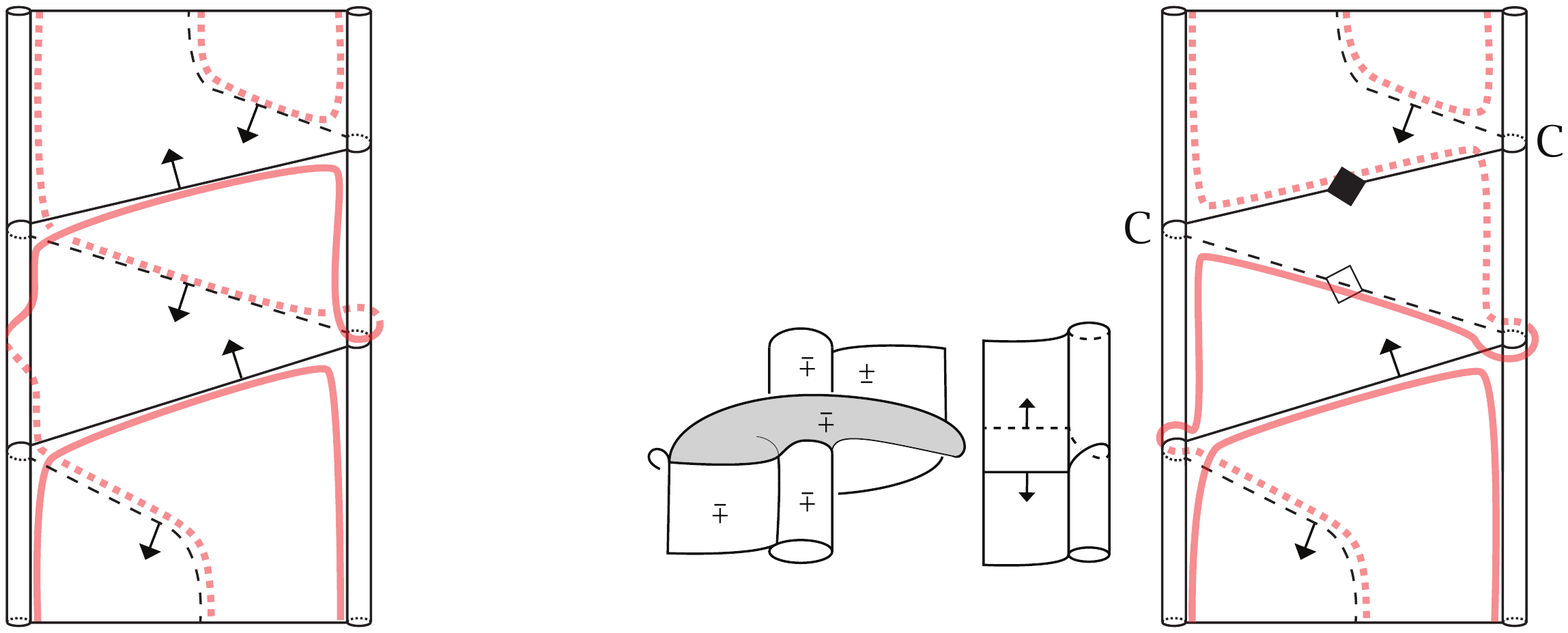}}
\caption{Comparison of the sutures in the complements of $B^G(\mathcal{S}_0)$ and $B(\mathcal{S}_0')$.}\label{suture comparison 2}
\end{figure}

\bibliographystyle{amsplain}

\end{document}